\documentclass{compositio}

\usepackage{amsmath,amssymb,amsthm} 
\usepackage[all]{xy}
\usepackage{extarrows}
\usepackage{graphicx}
\usepackage{bm} 
\usepackage{comment, url}
\usepackage[normalem]{ulem}
\usepackage{multicol}

\usepackage[pagewise]{lineno}

\usepackage{time}

\newtheorem{thm}{Theorem}[section]
\newtheorem{lem}[thm]{Lemma}
\newtheorem{cor}[thm]{Corollary}
\newtheorem{prop}[thm]{Proposition}  

\theoremstyle{remark}

\theoremstyle{definition}
\newtheorem{defn}[thm]{Definition}
\newtheorem{rem}[thm]{Remark} 

\newtheorem{eg}[thm]{Examples}

\newtheorem{prob}[thm]{Problem}
\newtheorem{def/prop}[thm]{Definition/Proposition}

\newtheorem{thmx}{Theorem}

\newtheorem{thmy}{Theorem}

\newtheorem{py}[thmy]{Problem\!} 

\numberwithin{equation}{section}

\usepackage[usenames]{color}

\def\tr{\mathop{\mathrm{tr}}\nolimits}
\def\det{\mathop{\mathrm{det}}\nolimits}
\def\Im{\mathop{\mathrm{Im}}\nolimits}
\def\Ker{\mathop{\mathrm{Ker}}\nolimits}

\def\Hom{\mathop{\mathrm{Hom}}\nolimits}

\def\Gal{\mathop{\mathrm{Gal}}\nolimits}

\def\SL{\mathop{\mathrm{SL}}\nolimits}
\def\GL{\mathop{\mathrm{GL}}\nolimits}
\def\M{\mathop{\mathrm{M}}\nolimits}

\def\mod{\mathop{\mathrm{mod}}\nolimits}

\newcommand{\mf}[1]{{\mathfrak{#1}}}

\newcommand{\bb}[1]{{\mathbb{#1}}}
\newcommand{\mca}[1]{{\mathcal{#1}}}
\newcommand{\bs}[1]{{\boldsymbol{#1}}}

\newcommand{\inj}{\hookrightarrow}
\newcommand{\surj}{\twoheadrightarrow}

\newcommand{\congto}{\overset{\cong}{\to}}

\newcommand{\iffu}{\underset{\rm iff}{\iff}}

\newcommand{\Z}{\bb{Z}}
\newcommand{\Zp}{\bb{Z}_{p}}
\newcommand{\Q}{\bb{Q}}

\newcommand{\R}{\bb{R}}
\newcommand{\C}{\bb{C}}

\newcommand{\F}{\bb{F}}

\newcommand{\der}{\partial}

\newcommand{\ol}{\overline}
\newcommand{\ul}{\underline}
\newcommand{\tc}[1]{\textcircled{\scriptsize {#1}}}
\newcommand{\ds}{\displaystyle}

\newcommand{\wt}[1]{{\widetilde{#1}}}

\DeclareMathOperator*{\restprod}%
 {\mathchoice{\ooalign{\ensuremath{\displaystyle\prod}\crcr\ensuremath{\displaystyle\coprod}}}%
             {\ooalign{\ensuremath{\textstyle\prod}\crcr\ensuremath{\textstyle\coprod}}}%
             {\ooalign{\ensuremath{\scriptstyle\prod}\crcr\ensuremath{\scriptstyle\coprod}}}%
             {\ooalign{\ensuremath{\scriptscriptstyle\prod}\crcr\ensuremath{\scriptscriptstyle\coprod}}}%
 }

\newcommand{\pmx}[1]{\begin{pmatrix}#1\end{pmatrix}}
\newcommand{\spmx}[1]{{\small \pmx{#1}}}

\def\be { \begin{equation} }
\def\ee { \end{equation} }

\title[Non-acyclic ${\rm SL}_2$-representations of twists knots] 
{Non-acyclic ${\rm SL}_2$-representations of twist knots, $-3$-Dehn surgeries, and $L$-functions}

\author{Ryoto Tange} 
\email{rtange.math@gmail.com}
\address{Department of Mathematics, School of Education, Waseda University\\ 
1-104, Totsuka-cho, Shinjuku-ku, 169-8050, Tokyo, Japan} 

\author{Anh T. Tran} 
\email{att140830@utdallas.edu}
\address{Department of Mathematical Sciences, The University of Texas at Dallas\\ 
Richardson, TX 75080, USA} 

\author{Jun Ueki} 
\email{uekijun46@gmail.com}
\address{Department of Mathematics, School of System Design and Technology, Tokyo Denki University\\ 
5 Senju Asahi-cho, Adachi-ku, 120-8551, Tokyo, Japan}


\begin{document}
	
\begin{abstract} 
We study irreducible $\SL_2$-representations of twist knots. 
We first determine all non-acyclic $\SL_2(\C)$-representations, which turn out to lie on a line denoted as $x=y$ in $\R^2$. 
Our main tools are character variety, Reidemeister torsion, and Chebyshev polynomials. 
We also verify a certain common tangent property, which yields a result on the $L$-functions of universal deformations, that is, the orders of the associated knot modules. 
Secondly, we prove that a representation is on the line $x=y$ if and only if it factors through the $-3$-Dehn surgery, 
and is non-acyclic if and only if the image of a certain element is of order 3. 
Finally, we study absolutely irreducible non-acyclic representations $\ol{\rho}$ over a finite field with characteristic $p>2$, 
to concretely determine all non-trivial $L$-functions $L_{\bs \rho}$ of the universal deformations over complete discrete valuation rings. 
We show among other things that $L_{\bs \rho}$ $\dot{=}$ $k_n(x)^2$ holds for a certain series $k_n(x)$ of polynomials. 
\end{abstract} 

\ \\[-2cm] 

\maketitle

\setcounter{tocdepth}{3}
{\small 
\tableofcontents } 

\newpage 

\setcounter{section}{-1}
\section{Introduction} 
In this article, we study irreducible $\SL_2$-representations of twist knots to concretely determine all non-trivial $L$-functions of ${\rm SL}_2$-deformations, which may be seen as infinitesimal ${\rm SL}_2$-analogues of the classical Alexander polynomials, as well as to point out several interesting phenomena on the zeros of the acyclic torsion functions. 
We remark that the former half of this paper concerns ${\rm SL}_2(\C)$-representations, while the latter deals with representations over finite fields and certain profinite rings.  

In this introduction, we fix our conventions, state all theorems, outline our arguments in this article, and attach remarks. 
For each $n\in \Z$, \emph{the twist knot} $J(2,2n)$ is defined by the diagram below, 
the horizontal twists being right handed if $n$ is positive and left handed if negative. 
\begin{center} 
\includegraphics[width=4cm]{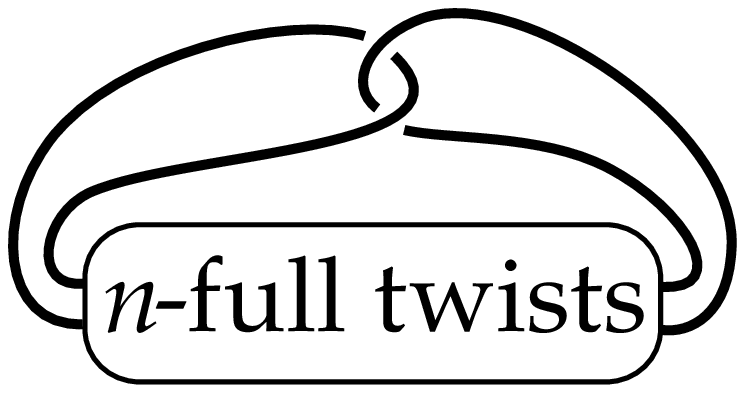} 
\end{center} 
We have $J(2,0)=0_1$ (unknot), $J(2,2)=3_1$ (trefoil), $J(2,4)=5_2$, and $J(2,-2)=4_1$ (figure-eight knot). 
Regarding a $1/2$-full twist to be a half twist, $J(2, -2n)$ and $J(2,2n+1)$ are the mirror images to each other, hence we only consider $J(2,2n)$.

The knot group $\pi_n:=\pi_1(S^3-J(2,2n))$ 
of $J(2,2n)$ admits the following presentation 
\cite[Proposition 1]{HosteShanahan2004JKTR}: 
\[\pi_n=\left<a,b \,|\, aw^n=w^nb\right>, \ \  	w=[a,b^{-1}]=ab^{-1}a^{-1}b.\]

Since twist knots are 2-bridge knots, the Culler--Shalen theory \cite{CullerShalen1983} together with Riley's calculation 
assures that conjugacy classes of $\rho\in \Hom(\pi_n, \SL_2(\C))$ are parametrized by 
$x:=\tr \rho(a)$ and $y:=\tr \rho(ab)$. 

A representation $\rho$ is said to be \emph{acyclic} if $H_i(\pi,\rho)=0$ holds for every $i$ and \emph{non-acyclic} if otherwise. 
In Section 1, we prove the following theorem. 
\begin{thmx}[(Non-acyclic representations)] \label{thm.na} \emph{
Conjugacy classes of non-acyclic irreducible $\SL_2(\C)$-representations of $J(2,2n)$ are exactly given by 
$x=y=1-2\cos\dfrac{2\pi l}{3n-1}$, $0<l \leq \dfrac{|3n-1|-1}{2}$, $l \in \Z$.}
\end{thmx}
This implies that every such representation corresponds to a point on the diagonal $x=y$ in $\R^2\subset \C^2$. 
In order to prove this assertion, we investigate the intersection of curves defined by Chebyshev-like polynomials $f_n(x,y)$, $\tau_n(x,y) \in \Z[x,y]$. 
The polynomial $f_n(x,y)$ defines a component of the character variety and coincides with the Riley polynomial $\Phi_n(x,u)$ via $-u=y-x^2+2$. 
The polynomial $\tau_n(x,y)$ is the Reidemeister torsion regarded as a function so that $\tau_n(x,y)=0$ iff a representation $\rho$ with $(\tr \rho(a),\tr\rho(ab))=(x,y)$ is non-acyclic. 
We first prove that the intersection of their zero set lies on $x=y$ and then determine all common roots of $f_n(x,x)$ and $\tau_n(x,x)$. 
We also introduce several Chebyshev-like polynomials $g_n,$ $h_n$, $k_n\in \Z[x]$ and prove $f_n(x,x)=g_nk_n$, $\tau_n(x,x)=h_nk_n$, where $k_n$ is the greatest common divisor. 
The roots of $g_n$ and $k_n$ are also determined and are given by ($3n-1$)-th roots of unity (cf. Proposition \ref{prop.roots} and Remark \ref{rem.3n-1}). 
We in addition prove the following theorem. 
\begin{thmx}[(The common tangent property)] \label{thm.tangent} \emph{
The two curves $f_n(x,y)=0$ and $\tau_n(x,y)=0$ in $\R^2$ have a common tangent line at every intersection point, 
while the second derivatives of their implicit functions do not coincide. 
In other words, every zero of $\tau_n(x,y)$ on $f_n(x,y)=0$ has multiplicity two in the function ring $\C[x,y]/(f_n(x,y))$.}
\end{thmx} 
At the end of Section 1, we define $\tau=\tau_n(x,y_f(x))=\tau_n(x_f(y),y)$ 
by using the implicit functions $y=y_f(x)$ and $x=x_f(x)$ of $f_n(x,y)=0$ around a non-acyclic representation 
and give explicit formulas of the 2nd derivatives (Corollary \ref{cor.d2}). 

In Section 2, we address a characterization problem of non-acyclic representations. 
We first prove the following theorem on the $-3$-Dehn surgeries, which is known to be exceptional. 
\begin{thmx}[($-3$-Dehn surgery)] \label{thm.Dehn} \emph{
The conjugacy class of an irreducible $\SL_2(\C)$-representation $\rho$ of $J(2,2n)$ is on the line $x=y$ if and only if $\rho$ factors through the $-3$-Dehn surgery. }
\end{thmx} 
In the proof, we invoke Riley's explicit representation and the relation formula of the eigenvalues at the meridian and the longitude. 
We also give an alternative proof for the only if part of Theorem \ref{thm.Dehn}, which is clearly workable over $\Z$, hence over a field with positive characteristic,  
by introducing doubly-indexed Chebyshev-like polynomials $\mf{d}_{m,n}(x,y)\in \Z[x,y]$ with $\mf{d}_{n,n}(x,x)-2=(x-2)(x+1)^2f_n(x,x)^2$ (Proposition \ref{prop.b}). 

We further prove the following theorem on a characterization of non-acyclic representations. 
\begin{thmx}[(Characterization)] \label{thm.order3} \emph{
The conjugacy class of an irreducible $\SL_2(\C)$-representation $\rho$ of $J(2,2n)$ on $x=y$ is non-acyclic if and only if $\rho(a^{-1}w^n)$ is of order 3. } 
\end{thmx} 

Our study is indeed motivated by a problem in arithmetic topology. 
In Section 3, based on our study of non-acyclic representations, 
we investigate the $L$-functions of universal deformations of residual representations. These $L$-functions were introduced in \cite{KMTT2018} in a perspective of the Hida--Mazur theory, and 
may be seen as infinitesimal ${\rm SL}_2$-analogues of the classical Alexander polynomials (cf. Remark \ref{rem.L}). 
Let $\ol{\rho}:\pi_n\to \SL_2(F)$ be a representation over a finite field $F$ with characteristic $p>2$ and let $O$ be a complete discrete valuation ring (CDVR) of characteristic zero with the residue field $F$. 
A \emph{deformation} (or a \emph{lift}) of $\ol{\rho}$ over a complete local $O$-algebra $R$ is a representation $\rho:\pi_n\to\SL_2(R)$ with the residual representation $\ol{\rho}$. A \emph{universal deformation} $\bs \rho:\pi_n\to \SL_2(\mca{R})$ of $\ol{\rho}$ 
over complete local $O$-modules (we also write ``over $O$'' for simplicity) 
is a deformation such that any deformation over any $R$ uniquely factors through $\bs \rho$ up to strict equivalence. 
If $\bs{\rho}$ is absolutely irreducible, then $\bs \rho$ uniquely exists up to $O$-isomorphism and strict equivalence. 

When $\mca{R}$ is a Noetherian UFD and the group homology $H_1(\pi_n,\bs{\rho})$ with local coefficients is a finitely generated torsion $\mca{R}$-module, the $L$-function $L_{\bs \rho} \in \mca{R}/$ $\dot{=}$ is defined to be the order of $H_1(\pi_n,\bs{\rho})$, where $\dot{=}$ denotes the equality up to multiplication by units in $\mca{R}$. 
Let $\Delta_{{\bs \rho},i}(t)$ denote the $i$-th $\bs \rho$-twisted Alexander polynomials. 
Then we have $L_{\bs \rho}$ $\dot{=}$ $\Delta_{{\bs \rho},1}(1)$. 
A general theory of twisted invariants yields $L_{\bs \rho} $ $\dot{=}$ $ \tau_{\bs \rho} \Delta_{{\bs\rho},0}(0)$. 
For most cases we have $\Delta_{{\bs\rho},0}$ $\dot{=}$ $1$, so that we have $L_{\bs \rho}$ $\dot{\neq}$ $1$ if and only if $\tau_{\ol{\rho}}=0$, that is, $\ol{\rho}$ is non-acyclic. 
Now B.~Mazur's Question 2 in \cite[page 440]{Mazur2000} may be varied as follows: 
\begin{py} \label{q} 
\emph{
Investigate the $L$-functions $L_{\bs\rho}$ of the universal deformations $\bs \rho$ over $O$ of absolutely irreducible non-acyclic residual representations $\ol{\rho}$.} 
\end{py} 

We completely answer this problem for the case of twist knots. Indeed,  
we will determine all residual representations with non-trivial $L$-functions, as well as explicitly determine the $L$-functions themselves and exhibit various examples. 
Among other things, we prove the following theorem, whose precise version is given as  Proposition \ref{prop.nap} (an analogue of Theorem \ref{thm.na} over a finite field) and Theorem \ref{thm.L.precise}. 

\begin{thmx} \label{thm.L} \emph{Notations being as above, every absolutely irreducible representation $\ol{\rho}:\pi_n\to \SL_2(F)$ of a twist knot corresponds to a root of $k_n$ in $F$.} 

\emph{
Suppose that $\ol{\rho}$ corresponds to a root $\ol{\alpha}$ of $k_n$ in $F$ with multiplicity $m$ and that 
$\alpha_1=\alpha,$ $\alpha_2, \cdots, \alpha_m$ are distinct lifts of $\ol{\alpha}$ with $k_n(\alpha_i)=0$ and $\alpha\in O$. 
If $\dfrac{\der f_n}{\der y}(\ol{\alpha},\ol{\alpha})\neq 0$ holds, so that there is a universal deformation ${\bs \rho}:\pi_n\to \SL_2(O[[x-\alpha]])$ over $O$, then the equalities 
\[L_{\bs \rho} \,\dot{=}\, k_n(x)^2 \,\dot{=}\, \prod_i (x-\alpha_i)^2\] 
in $\mca{R}=O[[x-\alpha]]$ hold. 
If in addition $p\nmid 3n-1$, then $m=1$ and $L_{\bs \rho}$ $\dot{=}$ $ (x-\alpha)^2$ holds. }

\emph{
If instead $\dfrac{\der f_n}{\der x}(\ol{\alpha},\ol{\alpha})\neq 0$, then a similar equality holds in $\mca{R}=O[[y-\alpha]]$. } 
\end{thmx} 

Since the character variety and the Reidemeister torsion are indeed defined over $\Z$, our arguments over $\C$ are mostly applicable over any integral domain. 
An analogue of Theorem \ref{thm.na} over $F$ determines all representations with nontrivial $L$-functions and that of 
Theorem \ref{thm.tangent} over $\mca{R}$ proves that every zero of the $L$-function is of order two. 

We remark that the common tangent property is also observed for non-acyclic representations of many other two-bridge knots, 
so that 
we may expect further geometric back ground for the order two phenomenon of the $L$-functions.


Our work in this article is derived from the scope of the following dictionary of the analogy between knots and prime numbers (cf.~\cite{MorishitaTerashima2007, MTTU, KMTT2018}, \cite[Chapter 14]{Morishita2012}).\\[-7mm] 
\begin{center}
\begin{tabular}{|c||c|}
\hline Low dimensional topology & Number theory\\ 
\hline \hline 
Deformation space of hyperbolic structures & Universal $p$-ordinary modular deformation space\\
\hline 
Dehn surgery points with $\Z$-coefficient & Arithmetic points \\
\hline 
\end{tabular}
\end{center} 
Arithmetic interpretations of 
exceptional surgeries and the diagonal $x=y$ would be of interest in our future study, 
while Theorem \ref{thm.order3} would help to investigate the images of representations. 
Our work would indeed live in the complement of the dictionary above in an extended picture. 
See also Remarks \ref{rem.Benard}, \ref{rem.Dehn}, \ref{rem.Brieskorn}, and \ref{rem.final}.\\

This article is organized as follows. 
All theorems are already stated in this introduction. 
In Section 1, we prove Theorems A and B on non-acyclic representations of twist knots.
In Subsections 1.1--1.3, we recall the Chebyshev polynomials $\mca{S}_n$, the defining polynomials $f_n(x,y)$ of the character varieties, and the Reidemeister torsions $\tau_n(x,y)$ respectively, and prepare several lemmas and propositions. 
In Subsection 1.4, we study the common roots of $f_n$ and $\tau_n$ to prove Theorem \ref{thm.na}. 
In Subsection 1.5, we introduce several Chebyshev-like polynomials $g_n$, $h_n$, and $k_n$ in $\Z[x]$ to describe the common divisors of $f_n(x,x)$ and $\tau_n(x,x)$. 
In Subsection 1.6, we calculate the slopes of tangent lines of the curves $f_n=0$ and $\tau_n=0$ to verify Theorem \ref{thm.tangent}. 
We also investigate the derivatives of $\tau=\tau_n(x,y_f(x))=\tau_n(x_f(y),y)$. 
In Subsection 1.7, we attach the tables of our polynomials and graphs with some observational remarks. 

In Section 2, we address a characterization problem of non-acyclic representations and prove Theorems \ref{thm.Dehn} and \ref{thm.order3}. 
In Subsection 2.1, we recall the Dehn surgery and remark some related results. 
In Subsection 2.2, we prove Theorem C by using eigenvalues method. 
In Subsection 2.3, we introduce doubly-indexed Chebyshev polynomials $\mf{d}_{m,n}$ to partially give an alternative proof of Theorem C. 
In Subsection 2.4, we recall Riley's representation and attach an outline of another partial proof. 
In Subsection 2.5, we introduce another Chebyshev-like polynomial $c_n$ to prove Theorem D and give a remark.

In Section 3, we study the $L$-functions of universal deformations. 
In Subsection \ref{subsec.twisted}, we overview the definitions and the relationships of Reidemeister torsions and twisted Alexander invariants. 
In Subsection \ref{subsec.resrep}, we study absolutely irreducible residual representations over a finite field $F$ with characteristic $p>2$. 
We prepare several fundamental tools and establish a version of Theorem \ref{thm.na} over $F$. 
In Subsection \ref{subsec.univdefo}, we investigate universal deformations ${\bs \rho}$ and point out the relationship between $\tau_{\bs \rho}$ and $\tau_n(x,y)$. 
In Subsection \ref{subsec.L}, we determine all residual representations with non-trivial $L$-functions, as well as determine the $L$-functions themselves. 
We obtain a precise version of Theorem \ref{thm.L} as a consequence of Theorem \ref{thm.tangent}. 
In Subsection \ref{subsec.eg}, we give a systematic study of examples of residual representations with non-trivial $L$-functions.

\section{Non-acyclic representations} 
\subsection{Chebyshev polynomials} 
We define the Chebyshev polynomial $\mca{S}_n(z)\in \Z[z]$ of the second kind for each $n\in \Z$ by $\mca{S}_n(2\cos \theta)=\dfrac{\sin n\theta}{\sin \theta}$ with $\theta \in \R$, or equivalently, by $\mca{S}_0(z)=0$, $\mca{S}_{\pm 1}(z)=\pm1$, 
and the bi-monic Chebyshev recurrence formula $\mca{S}_{n+1}(z)-z \mca{S}_n(z)+\mca{S}_{n-1}(z)=0$. 
We in addition have $\mca{S}_{\pm2}(z)=\pm z,$ $\mca{S}_{\pm3}(z)=\pm (z^2- 1)$, $\ldots$ . 
(We prefer to use $\mca{S}_n$ instead of $S_n=\mca{S}_{n+1}$ in the references, since the symmetricity $\mca{S}_{-n}=-\mca{S}_n$ simplifies our argument.) 
We may easily verify the following lemmas.  

\begin{lem} \label{lem.Cheb} 
{\rm (1)} We have $\mca{S}_n(\pm2)=(\pm1)^{n-1}n$ for every $n\in \Z$. 
Suppose $z\neq \pm 2$ and $z=s+s^{-1}$ for some $s\in \C^*$. Then we have $\mca{S}_n(z)=(s^n-s^{-n})/(s-s^{-1})$ for every $n\in \Z$. 

{\rm (2)} Let $R$ be any commutative ring with $z\in R$ and suppose that $(p_n)_n \in R^\Z$ satisfies the recurrence formula $p_{n+1}-zp_n+p_{n-1}=0$. 
Then we have $p_n=p_1\mca{S}_n(z)-p_0\mca{S}_{n-1}(z)=p_0\mca{S}_{n+1}(z)-p_{-1}\mca{S}_n(z)$ for every $n\in \Z$.
\end{lem} 

\begin{lem} \label{lem.Cheb2}
We have $\mca{S}_{n+1}(z)^2+\mca{S}_n(z)^2-z\mca{S}_{n+1}(z)\mca{S}_n(z)=1$ for every $n\in \Z$. 
\end{lem} 


\subsection{Character variety} 
We recall the description of the $\SL_2(\C)$-character variety of $J(2,2n)$ with use of the Chebyshev polynomials by following \cite{NagasatoTran2016OJM}. 

For each $g\in \pi_n$, let $\tr g$ denote the map $\Hom(\pi_n, \SL_2(\C))\to \C; \rho\mapsto \tr\rho(g)$.  
Then we have $x=\tr a=\tr b$, $y=\tr ab =\tr ba$. We also put $z=\tr w=\tr ab^{-1}a^{-1}b$. 

\begin{lem} We have $z=2x^2-x^2y+y^2-2$, hence $z-2=-(y-2)(x^2-y-2)$. 
\end{lem} 

\begin{proof} 
We have $\tr AB=\tr BA =\tr A \tr B -\tr AB^{-1}$ for any $A,B \in \SL_2(\C)$, hence 
$\tr A^{-1}=\tr A$ and $\tr A^2=(\tr A)^2-2$. These formulas yield 
$z=\tr w=\tr ab^{-1}a^{-1}b
=\tr bab^{-1}a^{-1}
=\tr ba \tr b^{-1}a^{-1} -\tr baab
=y^2-\tr a \tr ab^2+\tr b^2
=y^2-x (\tr ab \tr b-\tr a)+(x^2-2)
=y^2-x(xy-x)+(x^2-2)
=y^2-x^2y+2x^2-2
=2x^2-x^2y+y^2-2$. 
\end{proof} 

We put $f_n(x,y)=(y-1)\mca{S}_n(z)-\mca{S}_{n-1}(z)$ for each $n\in \Z$. 
Then we have $f_0=1$, $f_1=y-1$, and $f_{n+1}-zf_n+f_{n-1}=0$. 
The following proposition slightly refines the arrangement in 
\cite[(1.2), (1.3)]{NagasatoTran2016OJM} of Le's Theorem \cite[Theorem 3.3.1]{Le1993}. 
(Note that $K_{-n}=J(2,2n)$ holds for $K_{-n}$ in \cite{NagasatoTran2016OJM}.) 

\begin{prop} \label{prop.f} 
The $\SL_2(\C)$-character variety of $\pi_n$ is given by $(x^2-y-2)f_n(x,y)=0$. 
Namely, $(x,y)\in \C^2$ satisfies $(x^2-y-2)f_n(x,y)=0$ if and only if $(x,y)=(\tr \rho (a),\tr \rho(ab))$ holds for some representation $\rho$. 
Each conjugacy class of irreducible representations corresponds to each point on $f_n(x,y)=0$ with $x^2-y-2\neq0$. 
\end{prop} 

\begin{proof} 
By Le's Theorem \cite[Theorem 3.3.1]{Le1993}, 
the $\SL_2(\C)$-character variety of $\pi_n$ is given by $\wt{f}_n=\tr aw^nb^{-1}-\tr w^n=\tr b^{-1}aw^n-\tr w^n$. 
Since $\tr gw^{n+1}=\tr gw^n\tr w-\tr gw^{n-1}=z\tr gw^n- \tr gw^{n-1}$ holds for any $g$, 
we have $\wt{f}_{n+1}-z\wt{f}_n+\wt{f}_{n-1}=0$ for any $n\in \Z$. 
In addition, we have  
$\wt{f}_0=\tr ab^{-1}-2=\tr a\tr b^{-1}-\tr ab-2=x^2-y-2=(x^2-y-2)f_0$ and 
$\wt{f}_1=\tr awb^{-1}-\tr w=\tr ab^{-1}-z=x^2-y-z
=(x^2-y-2)-(z-2)=(x^2-y-2)(y-1)=(x^2-y-2)f_1.$ 
Therefore we have $\wt{f_n}=(x^2-y-2)f_n$ for any $n\in \Z$. 
Again by \cite[Theorem 3.3.1]{Le1993}, the divisor $x^2-y-2$ corresponds to the curve of reducible representations, 
while $f_n$ coincides with the Riley polynomial $\Phi_n(x,u)$ via $-u=y-x^2+2$. 
\end{proof}

\subsection{Reidemeister torsion} 
We recall the Reidemeister torsion and prove the recurrence formula. 

For an acyclic irreducible representation $\rho:\pi_n\to\SL_2(\C)$, the Reidemeister torsion $\tau_\rho:=\tau_\rho(S^3-J(2,2n))$ is defined in a usual way. 
We just remark that our convention precisely follows that in \cite[Section 4]{Tran2016TJM}. 
This $\tau_\rho$ extends to a function $\tau_n=\tau_n(x,y)$ over the character variety such that we have $\tau_n(x,y)=0$ if and only if a representation $\rho$ with $(\tr \rho(a),\tr\rho(ab))=(x,y)$ is non-acyclic. 
This function $\tau_n(x,y)$ coincides with the Reidemeister torsion of Riley's universal representation (cf.~Subsection \ref{subsec.univdefo}). 
Such convention also plays an important role, for instance, in the study of hyperbolic deformation (cf.~\cite[Section 3.2]{Porti2018ALM}). 

The second author proved the following formula (1) in \cite[Theorem 1]{Tran2016TJM}, whose assumption $x\neq 2$ there may be removed by Kitano's result \cite{Kitano1996PJM}, and remarked that we indeed have $\tau_n \in \Z[x,y]$.  (This formula generalizes to all genus one two-bridge knots \cite[Theorem 2]{Tran2018KMJ}.) 

\begin{prop} \label{prop.Tran} 
{\rm (1)} The Reidemeister torsion $\tau_n=\tau_\rho(S^3-J(2,2n))$ for a representation $\rho$ with $(\tr \rho(a),\tr \rho(ab))=(x,y)$ is given by 
\[\tau_n=(2-x)\dfrac{\mca{S}_{n+1}(z)-\mca{S}_{n-1}(z)-2}{z-2}+x\mca{S}_{n}(z)
=2\dfrac{(z-x)\mca{S}_n(z)+(x-2)(\mca{S}_{n-1}(z)+1)}{z-2}\]
in $\Z[x,y]$. If $z=2$, then $\tau_n=(2-x)n^2+xn$ holds. 

{\rm (2)} The sequence of polynomials $\tau_n \in \Z[x,y]$ is determined by $\tau_0=0$, $\tau_1=2$, and the recurrence formula $\tau_{n+1}-z\tau_{n}+\tau_{n-1} - 2(x-2)=0$ for $n\in \Z$. 
\end{prop} 

\begin{proof} We prove (2). By definition, we have $\tau_0=0$ and $\tau_1=2$. 
We first suppose $z\neq 2$. Then the sequence $(\tau_n+\dfrac{2(2-x)}{z-2})_n$ is a linear combination of Chebyshev polynomials and obeys the Chebyshev recurrence formula $(\tau_{n+1}+\dfrac{2(2-x)}{z-2})-z(\tau_n+\dfrac{2(2-x)}{z-2})+(\tau_{n-1}+\dfrac{2(2-x)}{z-2})=0$, which yields $\tau_{n+1}-z\tau_{n}+\tau_{n-1} - 2(x-2)=0$. 

If we instead suppose $z=2$, then by $\tau_n=(2-x)n^2+xn$, we have 
$\tau_{n+1}-2\tau_n+\tau_{n-1}-2(x-2)=
(2-x)((n+1)^2-2n^2+(n-1)^2)+x((n+1)-2n+(n-1)-2(x-2)=0$. 
\end{proof}

\subsection{Non-acyclic representations} 
 
Here we prove Theorem \ref{thm.na}. 
Note that Propositions \ref{prop.f} and \ref{prop.Tran} yield the following equivalences respectively: 

$f_n(x,y)=0 \iffu (y-1)\mca{S}_n(z)=\mca{S}_{n-1}(z)\cdots \tc{1},$

$\tau_n(x,y)=0 \iffu 
\left\{ \begin{array}{l}(z-x)\mca{S}_n(z)+(x-2)\mca{S}_{n-1}(z)+(x-2)=0 \cdots \tc{2} \text{\ \ if\ } z\neq 2,\\
(2-x)n^2+xn=0 \cdots \tc{2}' \text{\ \ if\ } z=2.\end{array} \right.$ 

\begin{lem} \label{lem.neq} 
If $f_n(x,y)=\tau_n(x,y)=0$, then we have $n\neq 0,1$ and $z\neq \pm2$. 
If in addition $x=y$, then we have $x\neq -1, 0, 2, 3$. 
\end{lem}

\begin{proof} If $n=0$, then $\tc{1}$ yields $0=-1$.  
If $n=1$ and $z\neq 2$, then $\tc{2}$ yields $z-2=0$. 
If $n=1$ and $z=2$, then $\tc{2}'$ yields $2=0$.
Since we obtain a contradiction in every case, we have $n\neq 0,1$. 

Suppose $z-2=-(y-2)(x^2-y-2)=0.$ Recall $\mca{S}_n(2)=n$. If $y-2=0$, then $f_n(x,y)=0$ yields $n=n-1$, hence a contradiction.  
If $x^2-y-2=0$, then $f_n(x,y)=0$ and $\tau_n(x,y)=0$ yield $(x^2-3)n=n-1$ and $(2-x)n^2+xn=0$,
hence $x=-1$ and $3n-1=0$, contradicting $n\in \Z$. 

Suppose instead $z+2=2x^2-x^2y+y^2=0$. Recall $\mca{S}_n(-2)=(-1)^{n-1}n$. 
By $f_n(x,y)=0$, we have $yn-1=0$, hence $y=\dfrac{1}{n}$. 
If $n$ is even, then $\tau_n(x,y)=0$ yields $x=\dfrac{2}{1-n}$ and  $z+2=\dfrac{(3n-1)^2}{n^2(n+1)^2}\neq0$. 
If $n$ is odd, then $\tau_n(x,y)=0$ yields $x=0$ and $z+2=\dfrac{1}{n^2}\neq 0$. Hence contradiction.  

If we assume $x-y=0$, then we have $z-2=-(x-2)^2(x+1)$ and $z+2=x^2(3-x)$. By $z\neq \pm2$, we have $x\neq -1,0,2,3$. 
\end{proof}

\begin{lem} \label{lem.1-x} 
Suppose $z\neq \pm 2$ and $x=y$. Then by writing $1-x=t+t^{-1}$ for $t\in \C^*$, we have $z=t^3+t^{-3}$, hence $\mca{S}_n(z)=(t^{3n}-t^{-3n})/(t^3-t^{-3})$.  
\end{lem} 
\begin{proof}  
Note that $z=-x^3+3x^2-2=3(x-1)-(x-1)^3=-3(t+t^{-1})-(-(t+t^{-1}))^3=t^3+t^{-3}$ and put $s=t^3$. Then Lemma \ref{lem.Cheb} (1) yields the assertion. 
\end{proof} 

The following proposition is the former half of Theorem \ref{thm.na}. 
\begin{prop} \label{prop.x=y} For every $n\neq 0,1$, 
the intersection of $f_n(x,y)=0$ and $\tau_n(x,y)=0$ lies on the line $x=y$ in $\R^2$. 
\end{prop} 

\begin{proof} Suppose $f_n(x,y)=\tau_n(x,y)=0$. Then the equalities \tc{1} and \tc{2} yield $(y-2)(x^2-x-y)\mca{S}_n(z)=x-2 \cdots \tc{3}$. By \tc{1} and Lemma \ref{lem.Cheb2}, we have $(1+(y-1)^2-z(y-1))\mca{S}_n(z)^2=1 \cdots \tc{4}$. 
Now \tc{3} and \tc{4} yield $(x-2)^2(1+(y-1)^2-z(y-1))=((y-2)(x^2-x-y))^2$. Hence $(y-2)(x-y)^2(x^2-y-2)=0$, that is, $(z-2)(x-y)^2=0$. By Lemma \ref{lem.neq}, we obtain $x=y$. 
\end{proof}

For each $v\in \R$, let $\lfloor v \rfloor$ 
denote the largest integer less than or equal to $v$. 
The following proposition is the latter half of Theorem \ref{thm.na}. 

\begin{prop} \label{prop.roots} For every $n\in \Z$, the roots of $f_n(x,x)$ are given by $x=1-2\cos \dfrac{l\pi}{3n-1}$ for $l \in \Z$ with $0<l<|3n-1|$. 
The common roots of $f_n(x,x)$ and $\tau_n(x,x)$ are given by $x=1-2\cos \dfrac{2l\pi}{3n-1}$ for $l \in \Z$ with $0<l\leq \dfrac{|3n-1|-1}{2}$. There are $\lfloor\dfrac{|3n-1|-1}{2}\rfloor$ common roots, and they are all real numbers. 
\end{prop} 

\begin{proof} 
By Lemmas \ref{lem.neq} and \ref{lem.1-x}, if we write $1-x=t+t^{-1}$ with $t\in \C^*$, then we have $z=s+s^{-1} \neq \pm2$ for $s=t^3$, 
\[\mca{S}_n(z)=\dfrac{s^n-s^{-n}}{s-s^{-1}}=\dfrac{t^{3n}-t^{-3n}}{t^3-t^{-3}},\ 
\text{and}\ \mca{S}_{n-1}(z)=\dfrac{s^{n-1}-s^{-(n-1)}}{s-s^{-1}}=\dfrac{t^{3(n-1)}-t^{-3(n-1)}}{t^3-t^{-3}}.\]
Hence $f_n(x,x)=(x-1)\mca{S}_n(z)-\mca{S}_{n-1}(z)=-\dfrac{t^{3n-1}-t^{-3n+1}}{t-t^{-1}}=-\mca{S}_{3n-1}(1-x)$. 
Therefore $f_n(x,x)=0$ holds iff $(t^{3n-1}+1)(t^{3n-1}-1)=0$ holds under the assumption $z\neq \pm 2$. 
Since $x\neq -1,3$ (Lemma \ref{lem.neq}), the roots of $f_n(x,x)$ are given by $\ds x=1-2\cos \frac{l\pi}{3n-1}$ for $l\in \Z$ with $0<l< |3n-1|$. 

Now by \tc{3}, such $t$ corresponds to a common root of $f_n(x,x)$ and $\tau_n(x,x)$ if and only if $(x^2-2x)\mca{S}_n(z)=1$ holds, that is, $(t^2+t^{-2}+1)(t^{3n}-t^{-3n})=t^3-t^{-3}$ $\cdots$ \tc{5} holds. 
If $t^{3n-1}-1=0$, then this equality \tc{5} holds. If $t^{3n-1}+1=0$, then \tc{5} implies $t^6=1$, 
contradicting $z\neq \pm2$ (Lemma \ref{lem.neq}). 
By $x\neq -1,0,2,3$ (Lemma \ref{lem.neq}), all common roots of $f_n(x,x)$ and $\tau_n(x,x)$ are given by 
$x=1-(t+t^{-1})=1-2\cos \dfrac{2\pi l}{3n-1}$ for $l \in \Z$ with $0< l < \dfrac{|3n-1|}{2}$. 
\end{proof}

\begin{proof}[Proof of Theorem \ref{thm.na}] 
Now the assertion follows from Propositions \ref{prop.x=y} and \ref{prop.roots}. 
\end{proof} 

\begin{rem} \label{rem.3n-1} 
The proof of Theorem \ref{thm.na} is also applicable to the case where $\C$ is replaced by any field $C$ with ${\rm char}=0$, 
by regarding $2\cos \dfrac{l\pi}{3n-1}$ to be $\zeta_{3n-1}^l+\zeta_{3n-1}^{-l}$ for a fixed primitive $(3n-1)$-th root of unity $\zeta_{3n-1}$ in an algebraic closure $\ol{C}$ of $C$. 

A version of Theorem \ref{thm.na} over a finite field will be given in Section \ref{sec.L}. 
\end{rem}

\subsection{Common divisors} 
Here we study the greatest common divisors $k_n$ of $f_n(x,x)$ and $\tau_n(x,x)$ by introducing several Chebyshev-like polynomials. 
Propositions \ref{prop.k1} and \ref{prop.k2} 
clarify that Theorem \ref{thm.na} exactly detects all common roots. 

\begin{prop} \label{prop.k1} Define $k_n$, $g_n$, $h_n\in \Z[x]$ for $n\in \Z$ by 
\[k_0=-1,\ k_1=1,\ g_0=-1,\ g_1=x-1,\ h_0=0,\ h_1=1,\]
\[k_{2n}-(x^2-3x)k_{2n-1}+k_{2n-2}=0,\ k_{2n+1}+xk_{2n}+k_{2n-1}=0,\]
\[g_{2n}+xg_{2n-1}+g_{2n-2}=0,\ g_{2n+1}-(x^2-3x)g_{2n}+g_{2n-1}=0,\]
and the same recurrence formula as $g_n$ for $h_n$. 
Then $f_n(x,x)=g_nk_n$ and $\tau_n(x,x)=2h_nk_n$ hold for every $n\in \Z$. 
\end{prop} 

\begin{proof} 
Here we write $f_n=f_n(x,x)$ and $\tau_n=\tau_n(x,x)$. 
We have $g_nk_n=f_n$ and $2h_nk_n=\tau_n$ for $n=-1,0,1$ (See also the table in Subsection \ref{subsec.tables}). 
Suppose that the assertion holds if $n$ is replaced by $n\mp1, n\mp2, n\mp3$. 
Note that $z=z|_{x=y}=-x^3+3x^2-2=-(x-1)(x^2-2x-2)$. 

\ul{If $n$ is even,} 
we have 
\begin{eqnarray*}
f_{n\mp1}&=&zf_{n\mp2}-f_{n\mp 3}\\
&=&zf_{n\mp2}-g_{n\mp3}k_{n\mp3}\\
&=&zf_{n\mp2}-((x^2-3x)g_{n\mp2}-g_{n\mp1})(-xk_{n\mp2}-k_{n\mp1})\\
&=&zf_{n\mp2}+x(x^2-3x)f_{n\mp2}-f_{n\mp1}+(x^2-3x)g_{n\mp2}k_{n\mp1}- xg_{n\mp1}k_{n\mp2}. 
\end{eqnarray*}
Since $z+x(x^2-3x)=-2$, we obtain
\[2f_{n\mp1}+2f_{n\mp2}=(x^2-3x)g_{n\mp2}k_{n\mp1}- xg_{n\mp1}k_{n\mp2}.\]
Hence we obtain 
\begin{eqnarray*}
g_nk_n&=&(-xg_{n\mp1}-g_{n\mp2})((x^3-3x)k_{n\mp1}-k_{n\mp 2})\\
&=&-x(x^3-3x)f_{n\mp1}+f_{n\mp2}-(2f_{n\mp1}+2f_{n\mp2})\\
&=&zf_{n\mp1}-f_{n\mp2}\\&=&f_n. 
\end{eqnarray*}

\ul{If $n$ is odd,} we have 
\begin{eqnarray*}
f_{n\mp1}&=&zf_{n\mp2}-f_{n\mp 3}\\
&=&zf_{n\mp2}-g_{n\mp 3}k_{n\mp3}\\
&=&zf_{n\mp2}-(-xg_{n\mp2}-g_{n\mp1})((x^2-3x)k_{n\pm2}-k_{n\mp1})\\
&=&zf_{n\mp2}+x(x^2-3x)f_{n\mp2}-f_{n\mp1}- xg_{n\mp2}k_{n\mp1}+ (x^2-3x)g_{n\mp1}k_{n\pm2},
\end{eqnarray*}
hence 
\[2f_{n\mp1}+2f_{n\mp2}=- xg_{n\mp2}k_{n\mp1} + (x^2-3x)g_{n\mp1}k_{n\pm2}.\]
Hence we obtain 
\begin{eqnarray*}
g_nk_n&=&((x^2-3x)g_{n\mp1}-g_{n\mp 2})(-xk_{n\mp1}-k_{n\mp2})\\
&=&-x(x^2-3x)f_{n\mp1}+f_{n\mp2}-(2f_{n\mp1}+2f_{n\mp2})\\
&=&zf_{n\mp1}-f_{n\mp2}\\&=&f_n. 
\end{eqnarray*}
Therefore $f_n=g_nk_n$ holds for every $n \in \Z$. 

Since $g_n$ and $h_n$ obey the same recurrence formula, by slightly modifying the argument above, we obtain $\tau_n=2h_nk_n$. 
For instance, suppose that $n$ is even and that the assertion holds if $n$ is replaced by $n\mp1, n\mp2, n\mp3$. 
Then we have $\tau_{n\mp 1}=z\tau_{n\mp2}-\tau_{n\mp3}\mp2(x-2)$, 
$\tau_{n\mp1}+\tau_{n\mp2}=(x^2-3x)h_{n\pm2}k_{n\mp2}-h_{n\pm1}k_{n\mp1}$, 
and $2h_nk_n=z\tau_{n\mp1}-\tau_{n\mp2}\mp2(x-2)=\tau_n$. 
We obtain the assertion for odd $n$ in a similar way. 
\end{proof} 

\begin{prop} \label{prop.k2}
We have $k_{2n}=\mca{S}_{3n}(1-x)+\mca{S}_{3n-1}(1-x)$ and $k_{2n+1}=\mca{S}_{3n+1}(1-x)$ for $n\in \Z$. 
This $k_n$ is the greatest common divisor of $f_n(x,x)$ and $\tau_n(x,x)$. 
\end{prop}

\begin{proof} 
We have $k_0(x)= -1= 0+(-1)=\mca{S}_{0}(1-x)+ \mca{S}_{-1}(1-x)$, 
$k_1(x)=1=\mca{S}_1(1-x)$. 
We may in addition verify that the right hand sides satisfy the recurrence formula for $k_n$ by using that for $\mca{S}_n$. 
The GCD property is obtained by Proposition \ref{prop.roots} and the following equalities (cf.~\cite[Lemma 4.13]{LeTran2015}): 
\[\ds \mca{S}_n(z)=\prod_{0< l< n}(z-2\cos\dfrac{l \pi}{n}), \ \ 
\ds \mca{S}_n(z)+\mca{S}_{n-1}(z)=\prod_{0< l<n}(z-2\cos\dfrac{2l\pi}{2n-1}),\]
and $\mca{S}_{-n}(z) = - \mca{S}_{n}(z)$ for every $n \in \Z_{>0}$. 
We may easily see that $$\ds k_n(x)={\rm sgn}(n)\prod_{0<l\leq \frac{|3n-1|-1}{2}} (1-2\cos \dfrac{2l\pi}{3n-1}-x)$$ holds for every $0\neq n \in \Z$, so that $k_n$ is the common divisor of $f_n(x,x)$ and $\tau_n(x,x)$. 
\end{proof}

The values of $k_n$ at $x=-1,0,2,3$ will be used in Section 3. 
\begin{cor} \label{lem.k2} 
For any $n\in \Z$, we have 
$k_{2n}(-1)=6n-1$, $k_{2n+1}(-1)=3n+1$, 
$k_{4n}(0)=k_{4n+3}(0)=-1$, $k_{4n+1}(0)=k_{4n+2}(0)=1$, 
$k_n(2)=(-1)^{n-1}$, 
$k_{2n}(3)=(-1)^{n-1}$, $k_{2n+1}(3)=(-1)^n(3n+1)$. 

If $n=2m$, then $k_n(-1)=3n-1$. If $n=2m+1$, then $k_n(-1)=\dfrac{3n-1}{2}$ and $k_n(3)=\pm \dfrac{3n-1}{2}$ hold. 
Other values in above are $\pm1$. 
\end{cor}

\begin{proof} 
If $x=-1,0,2,3$, then $1-x=2,1,-1,2$. 
Recall $\mca{S}_n(2\cos \theta)=\dfrac{\sin n\theta}{\sin \theta}$ and $\mca{S}(\pm2)=(\pm 1)^{n-1}n$. 
If $x=2$, we may take $\theta=\pi/3$ so that we have $2\cos \theta=1-x=-1$, hence 
$\mca{S}_{3n}(-1)=0$, $\mca{S}_{3n+1}(-1)=1$, $\mca{S}_{3n+2}(-1)=-1$, 
hence the previous proposition yields the assertion. 
If $x=0$, then we may take $\theta=\pi/6$ and a similar argument proves the assertion. The cases of $x=1,-3$ are straight forward. 
\end{proof}

\subsection{The common tangent property} \label{subsec.tangent}  
We prove the common tangent property (Theorem \ref{thm.tangent}), namely, that 
(i) the two curves $f_n(x,y)=0$ and $\tau_n(x,y)=0$ in $\R^2$ have a common tangent line at every intersection point $(\alpha,\alpha)$, 
while (ii) the second derivatives of their implicit functions at $(\alpha,\alpha)$ do not coincide. 

\begin{lem} \label{lem.der} 
The differential of the Chebyshev polynomial $\mca{S}_n=\mca{S}_n(z)$ for each $n\in \Z$ is given by the following equality in $\Z[z]$.
\[\frac{d \mca{S}_n}{d z} = \frac{(n-1) \mca{S}_{n+1} - (n+1)\mca{S}_{n-1}}{z^2-4}\]
\end{lem}

\begin{proof} 
Assume $z\neq \pm 2$ and put $z = s + s^{-1}$, so that we have $\mca{S}_n  = (s^{n} - s^{-n}) / (s - s^{-1})$. 
Since $\dfrac{d \mca{S}_n}{d z} \dfrac{d z}{ds}= \dfrac{d \mca{S}_n}{d s}$, a direct calculation yields 
\[
\frac{d \mca{S}_n}{d z}  = \frac{(n-1)(s^{n+1} - s^{-n-1}) - (n+1) (s^{n-1} - s^{1-n})}{(s - s^{-1})^3} =  \frac{(n-1) \mca{S}_{n+1} - (n+1)\mca{S}_{n-1}}{z^2-4}.
\]
Since the two sides are polynomials in $z$, they coincide also at $z=\pm2$. 
\end{proof} 

Recall $f_n(x,y)=(y-1)\mca{S}_n(z)-\mca{S}_{n-1}(z)$ for $z=2x^2-x^2y+y^2-2$. 
Then we have 
\[\dfrac{\der f_n}{\der x}=(y-1)\dfrac{\der \mca{S}_n}{\der x}-\dfrac{\der \mca{S}_{n-1}}{\der x}
=(y-1)\dfrac{d \mca{S}_n}{dz} \dfrac{\der z}{\der x}-\dfrac{d \mca{S}_{n-1}}{dz} \dfrac{\der z}{\der x},\]
\[\dfrac{\der f_n}{\der y}=\mca{S}_n + (y-1)\dfrac{\der\mca{S}_n}{\der y}-\dfrac{\der\mca{S}_{n-1}}{\der y}
=\mca{S}_n +(y-1)\dfrac{d \mca{S}_n}{dz} \dfrac{\der z}{\der y}-\dfrac{d\mca{S}_{n-1}}{dz} \dfrac{\der z}{\der y},\]
\[\dfrac{\der z}{\der x}=4x-2xy, \ \ \dfrac{\der z}{\der y}=-x^2+2y,\]
\[\dfrac{d\mca{S}_n}{dz}=\dfrac{(n-1)z\mca{S}_n-2n\mca{S}_{n-1}}{z^2-4}, \ \ 
\dfrac{d\mca{S}_{n-1}}{dz}=\dfrac{(2n-2)\mca{S}_n-nz\mca{S}_{n-1}}{z^2-4},\]
where the last line is obtained by Lemma \ref{lem.der} and $\mca{S}_{n+1}-z\mca{S}_n+\mca{S}_{n-1}=0$. 

In what follows, let $x=\alpha$ be a root of $k_n$, so that the point $(x,y)=(\alpha,\alpha)$ corresponds to a non-acyclic representation. 
Recall $z(\alpha,\alpha)=-\alpha^3+3\alpha^2-2\neq \pm2$, $\alpha\neq -1,0,2,3$ (Lemma \ref{lem.neq}), $\mca{S}_n=\dfrac{1}{\alpha^2-2\alpha}$, and $\mca{S}_{n-1}=\dfrac{\alpha-1}{\alpha^2-2\alpha}$ (\tc{3} in the proof of Proposition \ref{prop.x=y}). Putting all above together, we obtain 
\[\dfrac{\der f_n}{\der x}(\alpha,\alpha)=\dfrac{2((2n-1)\alpha^2+(-4n+2)\alpha+1)}{(\alpha+1)\alpha(\alpha-2)(\alpha-3)}, \ \ 
\dfrac{\der f_n}{\der y}(\alpha,\alpha)=\dfrac{2(n\alpha^2-2n\alpha-1)}{(\alpha+1)\alpha(\alpha-2)(\alpha-3)}.\]

The following lemma yields that for every root $x=\alpha$ of $k(x,x)$, the point $(\alpha,\alpha)$ is a regular point of $f_n(x,y)=0$. 

\begin{lem} \label{lem.n=-1} 
{\rm (1)} We have $n\alpha^2-2n\alpha-1=0$ if and only if $n=-1$ and $\alpha=1$ hold. 

{\rm (2)} We have $(2n-1)\alpha^2+(-4n+2)\alpha+1 \neq 0$. 
\end{lem} 
\begin{proof} 
Note that $\cos (r\pi) \in \Q$ and $r \in \Q$ imply $\cos (r\pi)=0,\pm\dfrac{1}{2},\pm1$. 
Indeed, suppose that $\zeta=e^{r\pi\sqrt{-1}}$ is a primitive $m$-th root of unity. 
Since $\zeta$ is a root of $z^2-2\cos (r\pi) z+1 \in \Q[z]$, the $m$-th cyclotomic polynomial is of degree one or two, hence Euler's totient function satisfies $\varphi(m)=1,2$, hence $m=1,2,3,4,6$. 

If $(n,\alpha)=(-1,1)$, then $n\alpha^2-2n\alpha-1=0$ holds. 
Conversely, suppose $n\alpha^2-2n\alpha-1=0$. 
Then we have $\alpha=1\pm\sqrt{1-\dfrac{1}{n}}=1-2\cos \dfrac{2l\pi}{3n-1}$ by Proposition \ref{prop.roots}.  
Hence $1-\dfrac{1}{n}=4\cos^2 \dfrac{2l\pi}{3n-1}=2+2\cos \dfrac{4l\pi}{3n-1} \in \Q$, 
hence $\cos (r\pi)=0,\pm\dfrac{1}{2},\pm1$ for $r=\dfrac{4l}{3n-1}$. 
By $\cos (r\pi)=0,\pm\dfrac{1}{2},\pm1$, we obtain $n=\pm1$. 
By $0,1\neq n \in \Z$, we obtain $n=-1$, hence $\alpha=1$. 

If $(2n-1)\alpha^2+(-4n+2)\alpha+1=0$, then we have $\alpha=1\pm\sqrt{1-\dfrac{1}{2n-1}}$ and a similar argument yield $n=0,1$, contradicting $n\neq 0,1$.  
\end{proof} 

Now suppose $(n,\alpha)\neq(-1,1)$ so that $\dfrac{\der f_n}{\der y}(\alpha,\alpha)\neq 0$, and let $y=y_f(x)$ denote the implicit function of $f_n(x,y)=0$ around the point $(\alpha,\alpha)$. Then the slope of the tangent line at $(\alpha,\alpha)$ on $f_n(x,y)=0$ is given by the differential of the implicit function as 
\[\ds \frac{dy_f}{dx} = -\frac{(\partial f_n/\partial x)(x,y_f(x))}{(\partial f_n/\partial y)(x,y_f(x))},\ \  \frac{dy_f}{dx}|_{x=\alpha}= - \frac{(2n-1)\alpha^2-(4n-2)\alpha+1}{n\alpha^2-2n\alpha-1} \cdots \tc{f}.\]
If instead $(n,\alpha)=(-1,1)$, then we may regard $\ds \frac{dy_f}{dx}|_{x=1}=\dfrac{\,2\,}{0}=\infty$.

Next, recall $\tau_n(x,y)=-2\dfrac{(z-x)\mca{S}_n(z)+(x-2)(\mca{S}_{n-1}(z)+1)}{z-2}$ in $\Z[x,y]$. 
By a similar calculation, the partial derivatives at $(\alpha,\alpha)$ are given by 
\[\dfrac{\der \tau_n}{\der x}(\alpha,\alpha)=\dfrac{2((2n-1)\alpha^2+(-4n+2)\alpha+1)}{(\alpha+1)\alpha(\alpha-2)^2(\alpha-3)},\ \ \dfrac{\der \tau_n}{\der y}(\alpha,\alpha)=\dfrac{2(n\alpha^2-2n\alpha-1)}{(\alpha+1)\alpha(\alpha-2)^2(\alpha-3)}.\]

Suppose $(n,\alpha)\neq (-1,1)$, so that $\dfrac{\der y_\tau}{\der y}(\alpha,\alpha)\neq 0$, and
let $y=y_{\tau}(x)$ denote the implicit function of $\tau_n(x,y)=0$ around the point $(\alpha,\alpha)$. Then we have 
\[\frac{dy_\tau}{dx} = -\frac{(\partial \tau_n/\partial x)(x,y_f(x))}{(\partial \tau_n/\partial y)(x,y_f(x))}, \ \ 
\frac{dy_\tau}{dx}|_{x=\alpha}= - \frac{(2n-1)\alpha^2-(4n-2)\alpha+1}{n\alpha^2-2n\alpha-1} \cdots \tc{$\tau$}.\] 
If instead $(n,\alpha)=(-1,1)$, then we may regard $\dfrac{dy_\tau}{dx}(1,1)=\dfrac{\,2\,}{0}=\infty$. 

Since $\tc{f}$ and $\tc{$\tau$}$ coincide, $f_n(x,y)=0$ and $\tau_n(x,y)=0$ have a common tangent line at every non-acyclic representation. 
Thus the first part of Theorem \ref{thm.tangent} is proved.\\ 

If $(n,\alpha)\neq (-1,1)$, then 
by $\dfrac{d^2 y_f}{dx^2} = \dfrac{d}{dx} 
\dfrac{-(\der f_n/\der x)(x,y_f(x))}{(\der f_n/\der y)(x,y_f(x))} 
$ 
and the similar for $y_\tau$, we obtain 
 \begin{eqnarray*}
\frac{d^2 y_f}{dx^2}|_{x=\alpha}
&=& \frac{3n-1}{(n\alpha^2-2n\alpha-1)^3} \big[ (2n^2-n)\alpha^5+(-4n^2+1)\alpha^4+(-8n^2+11n-2)\alpha^3 \\
&& + \, (16n^2-10n-3)\alpha^2+(-8n+2)\alpha+4) \big], 
\end{eqnarray*}
\begin{eqnarray*}
\frac{d^2 y_\tau}{dx^2}|_{x=\alpha} 
&=& \frac{3n-1}{(n\alpha^2-2n\alpha-1)^3} \big[(2n^2-n)\alpha^5 +(-4n^2+3n)\alpha^4 +(-8n^2-n+2)\alpha^3 \\
&& + \, (16n^2-7n-4)\alpha^2+(10n-4)\alpha+4 \big].
\end{eqnarray*}

The difference between the two second derivatives is given by 
\[\frac{d^2}{dx^2}(y_f-y_\tau)|_{x=\alpha}=\frac{-(3n-1)^2(\alpha+1)\alpha(\alpha-2)(\alpha-3)}{(n\alpha^2-2n\alpha-1)^3}.\]
Since $\alpha\neq -1,0,2,3$ (Lemma \ref{lem.neq}), we have $\ds \frac{d^2}{dx^2}(y_f- y_\tau)|_{x=\alpha}\neq 0$. 
Thus we obtain the second part of Theorem \ref{thm.tangent} for implicit functions in terms of $x$.\\ 

We may also obtain a similar result for the implicit functions in terms of $y$. 
Indeed, by Lemma \ref{lem.n=-1} (2), we have $\dfrac{\der f_n}{\der x}\neq 0$ and $\dfrac{\der \tau_n}{\der x}\neq 0$ at a non-acyclic representation $(\alpha,\alpha)$. 
Let $x=x_f(y)$ and $x=x_\tau(y)$ denote the implicit functions of $f_n(x,y)=0$ and $\tau_n(x,y)=0$ in terms of $y$ around $(\alpha,\alpha)$. Then a similar calculation yields 
\[\dfrac{d^2}{dy^2}(x_f-x_\tau)|_{y=\alpha}=
\dfrac{-(3n-1)^2(\alpha+1)\alpha(\alpha-2)(\alpha-3)}{((2n-1)\alpha^2+(-4n+2)\alpha+1)^3} \neq 0.\]

\begin{rem} 
We may verify all these calculations above by using Mathematica 12.0 \cite{Mathematica12}. 
If we instead use $u=x^2-y-2$, then we have $\dfrac{du}{dx} = 2x-\dfrac{dy}{dx}$ and $\dfrac{d^2u}{dx}=2-\dfrac{d^2y}{dx}$ for the implicit functions $u=u_f(x)$ and $u_\tau(x)$ of $f_n(x,y)=\Phi_n(x,u)=0$ and $\tau_n(x,y)=0$. 
\end{rem}

The following lemma is an elementary exercise of calculus. 

\begin{lem} \label{lem.mult2} 
Let $f=f(x,y)$ and $g=g(x,y)\in \C[x,y]$ and $(a,b)\in \C^2$ a common zero with multiplicity $m$.
Suppose $\dfrac{\der f}{\der y}(a,b)\neq 0$ and $\dfrac{\der g}{\der y}(a,b)\neq 0$, and let $y_f(x), y_g(x) \in \C[x]$ denote the implicit functions of $f=0$ and $g=0$ in a neighborhood of $(a,b)$. If $\dfrac{d}{dx}(y_f-y_g)(a,b)=0$ and $\dfrac{d^2}{dx^2}(y_f-y_g)(a,b)\neq0$, then $m=2$ holds, that is, we have $g \in (x-a,y-b)^2$ and $g \not\in (x-a,y-b)^3$ in $\C[x,y]/(f)$. 
\end{lem} 

\begin{proof} 
For here, we denote the partial derivatives by $f_x=\dfrac{\der f}{\der x}$, $f_{xy}=\dfrac{\der^2 f}{\der x \der y}$, and etc. 
Note that the implicit differentiation yields $\dfrac{d}{dx}y_f=-\dfrac{f_x}{f_y}$ and $\dfrac{d^2}{dx^2}y_f=-\dfrac{f_y^2f_{xx}-2f_xf_yf_{xy}+f_x^2f_{yy}}{f_y^3}$. 
Suppose that $(a,b)$ is a common zero of $f$ and $g$. Then for $k\in \C[x,y]$, we have the equivalences 
\begin{eqnarray*}
g+kf \in (x-a,y-b)^2 \subset \C[x,y] 
&\iff& (g+kf)_x(a,b)=0,\ (g+kf)_y(a,b)=0\\
&\iff& (g_x+k f_x)(a,b)=0,\ (g_y+k f_y)(a,b)=0,
\end{eqnarray*} 
which yield the equivalences 
\begin{eqnarray*}
g\in (x-a,y-b)^2 \subset \C[x,y]/(f) &\iff& g+kf \in (x-a,y-b)^2 \subset \C[x,y] \text{\ for\ some\ }k\in \C[x,y]\\
&\iff& \dfrac{d}{dx}(y_f-y_g)(a,b)=0. 
\end{eqnarray*}

Now suppose that $g\in (x-a,y-b)^3 \subset \C[x,y]/(f)$ and write $g+kf \in (x-a,y-b)^3$ for $k \in \C[x,y]$. 
Then at the point $(a,b)$, we have $(g+kf)_x=(g+kf)_y=0$ and $(g+kf)_{xx}=(g+kf)_{xy}=(g+kf)_{yy}=0$, hence $(g_x+kf_x)=(g_y+kf_y)=0$ and 
$(g_{xx} +2k_xf_x+kf_{xx})= (g_{yy} +2k_yf_y+kf_{yy})= (g_{xy}+k_xf_y+k_yf_x+kf_{xy})=0$. 
Hence we have 
\begin{eqnarray*}
\dfrac{d^2}{dx^2}(y_f-y_g)
&=&
-\dfrac{f_y^2f_{xx}-2f_xf_yf_{xy}+f_x^2f_{yy}}{f_y^3}
-(-\dfrac{g_y^2g_{xx}-2g_xg_yg_{xy}+g_x^2g_{yy}}{g_y^3})\\
&=&\dfrac{1}{g_y}(k(f_{xx}-2\dfrac{f_x}{f_y}f_{xy}+(\dfrac{f_x}{f_y})^2f_{yy})+(g_{xx}-2\dfrac{g_x}{g_y}g_{xy}+(\dfrac{g_x}{g_y})^2g_{yy}))\\
&=&\dfrac{1}{g_y}(-2k_xf_x+2\dfrac{f_x}{f_y}(k_xf_y+k_yf_x)-2(\dfrac{f_x}{f_y})^2k_yf_y) \ =0
\end{eqnarray*}
at $(a,b)$. Thus we obtain the assertion as the contraposition. 
\end{proof} 

\begin{proof}[Proof of Theorem \ref{thm.tangent}] 
As we calculated in above, if $\dfrac{\der f_n}{\der y}(\alpha,\alpha)\neq 0$, then $\dfrac{d}{dx}(y_f-y_\tau)|_{x=\alpha}=0$ and $\dfrac{d^2}{dx^2}(y_f-y_\tau)|_{x=\alpha}\neq 0$ hold. If instead $\dfrac{\der f_n}{\der x}(\alpha,\alpha)\neq 0$, then $\dfrac{d}{dy}(x_f-x_\tau)|_{y=\alpha}=0$ and $\dfrac{d^2}{dy^2}(x_f-x_\tau)|_{y=\alpha}\neq 0$. In both cases, Lemma \ref{lem.mult2} completes the proof. 
\end{proof} 

\begin{rem} \label{rem.Benard} 
Theorems \ref{thm.na} and \ref{thm.tangent} were initially observed for several examples in the study of $L$-functions in \cite{KMTT2018}, as well as pointed out by \cite[Remark 4.6]{Benard2020OJM}. 
L\'eo B\'enard suggests that we should consider representations of all twist knots as those of the Whitehead link. 
His approach will be extensively discussed in our forthcoming work. 
\end{rem} 

The following Corollaries of Theorem \ref{thm.tangent} and their proofs will play key roles in the study of $L$-functions in Subsection \ref{subsec.L}. 

\begin{cor} \label{cor.d2} 
Let the notation be as above, where $y_f$ is defined if $(n,\alpha)\neq (-1,1)$. 
For each point $(x,y)$ on the curve $f(x,y)=0$ in a small neighborhood of the point $(\alpha,\alpha)$ corresponding to a non-acyclic representation, define $\tau=\tau_n(x,y_f(x))=\tau_n(y_f(y),y)$. 
Then, we have 
\[\tau|_{x=\alpha}=\tau|_{y=\alpha}=0,\ \ \dfrac{d\tau}{dx}|_{x=\alpha}=0, \ \ \dfrac{d\tau}{dy}|_{y=\alpha}=0,\]
\[\dfrac{d^2\tau}{dx^2}=\dfrac{\der \tau_n}{\der y}\dfrac{d^2}{dx^2}(y_f-y_\tau), \ 
\dfrac{d^2\tau}{dx^2}|_{x=\alpha}=\dfrac{2(3n-1)^2}{(\alpha-2)(n\alpha^2-2n\alpha-1)^2} \neq 0,\]
\[\dfrac{d^2\tau}{dy^2}=\dfrac{\der \tau_n}{\der x}\dfrac{d^2}{dy^2}(x_f-x_\tau), \ 
\dfrac{d^2\tau}{dy^2}|_{y=\alpha}=
\dfrac{2(3n-1)^2}{(\alpha-2)((2n-1)\alpha^2+(-4n+2)\alpha+1)^2}\neq 0.\]
\end{cor} 

\begin{proof} 
The Leibniz rule for two variable function and Theorem \ref{thm.tangent} yield 
\[\dfrac{d\tau}{dx}=\dfrac{d}{dx}\tau_n(x,y_f(x))=\dfrac{\der \tau_n}{\der x}+\dfrac{\der \tau_n}{\der y}\dfrac{d y_f}{dx} 
=\dfrac{\der \tau_n}{\der y}(\dfrac{\der\tau_n/\der x}{\der \tau_n/\der y}+\dfrac{d y_f}{dx})
=\dfrac{\der \tau_n}{\der y}(\dfrac{d y_f}{dx}-\dfrac{d y_\tau}{dx})\underset{x=\alpha}=0.\]
Again by the Leibniz rule, the second derivatives of the equality $0=\tau_n(x,y_\tau(x))$ of functions in $x$ is given by 
\[0=\dfrac{d^2}{dx^2}\tau(x,y_\tau(x))
=\dfrac{\der^2 \tau_n}{\der x^2}+2\dfrac{\der^2 \tau_n}{\der x \der y}\dfrac{dy_\tau}{dx}+\dfrac{\der^2 \tau_n}{\der y^2}(\dfrac{dy_\tau}{dx})^2+\dfrac{\der \tau_n}{\der y}\dfrac{d^2 y_\tau}{dx^2}.\]
By this equality and a similar calculation, we obtain 
\[\dfrac{d^2\tau}{dx^2}=\dfrac{d^2}{dx^2}\tau_n(x,y_f(x))
=\dfrac{\der^2 \tau_n}{\der x^2}+2\dfrac{\der^2 \tau_n}{\der x \der y}\dfrac{dy_f}{dx}+\dfrac{\der^2 \tau_n}{\der y^2}(\dfrac{dy_f}{dx})^2+\dfrac{\der \tau_n}{\der y}\dfrac{d^2 y_f}{dx^2}=\dfrac{\der \tau_n}{\der y}\dfrac{d^2}{dx^2}(y_f-y_\tau),\]
whose value at $x=\alpha$ is given by 
\[\dfrac{2(n\alpha^2-2n\alpha-1)}{(\alpha+1)\alpha(\alpha-2)^2(\alpha-3)} \frac{-(3n-1)^2(\alpha+1)\alpha(\alpha-2)(\alpha-3)}{(n\alpha^2-2n\alpha-1)^3}=\dfrac{2(3n-1)^2}{(\alpha-2)(n\alpha^2-2n\alpha-1)^2} \neq 0.\] 

Similar calculations yield the equalities for $x_f$ and $x_\tau$. 
\end{proof} 

\begin{cor} \label{cor.tauk} 
Any root of $\tau_n(x,y_f(x))$ and $\tau_n(x_f(y),y)$ is a root of $k_n$. 
\end{cor} 

\begin{proof}
Let $x=\alpha$ be a root of $\tau_n(x,y_f(x))$ and put $\beta=y_f(\alpha)$. 
Then we have $f_n(\alpha,\beta)=0$ and $\tau_n(\alpha,\beta)=0$. Hence by Theorem \ref{thm.na}, we have $\alpha=\beta$ and $x=\alpha$ is a root of $k_n(x)$. 
A similar argument holds for $\tau_n(x_f(y),y)$. 
\end{proof}

\subsection{Tables and graphs} \label{subsec.tables} 
Here we attach the tables of our polynomials $f_n(x,x)$, $\tau_n(x,x)$, $g_n$, $h_n$, $k_n$ for $|n|\leq5$ and $f_n(x,y),$ $\tau_n(x,y)$ for $|n|\leq3$. 

\begin{center}
{\footnotesize 

\begin{tabular}{c||c}
$n$&$f_n(x,x)$\\
\hline  \hline
$-5$&$-(x - 1) (x^2 - 2 x - 1) (x^4 - 4 x^3 + 2 x^2 + 4 x - 1) (x^8 - 8 x^7 + 20 x^6 - 8 x^5 - 30 x^4 + 24 x^3 + 12 x^2 - 8 x - 1)$\\
$-4$&$(x^6 - 7 x^5 + 15 x^4 - 6 x^3 - 11 x^2 + 6 x + 1) (x^6 - 5 x^5 + 5 x^4 + 6 x^3 - 7 x^2 - 2 x + 1)$\\
$-3$&$-(x - 1) (x^2 - 3 x + 1) (x^2 - x - 1) (x^4 - 4 x^3 + x^2 + 6 x + 1)$\\
$-2$&$(x^3 - 4 x^2 + 3 x + 1) (x^3 - 2 x^2 - x + 1)$\\
$-1$&$-(x - 1) (x^2 - 2 x - 1)$\\ \hline
0&1\\
1&$x-1$\\ \hline
2&$-(x^2 - 3 x + 1) (x^2 - x - 1)$ \\
3&$(x - 1) (x^2 - 2 x - 1) (x^4 - 4 x^3 + 2 x^2 + 4 x - 1)$\\
4&$-(x^5 - 6 x^4 + 10 x^3 - x^2 - 6 x + 1) (x^5 - 4 x^4 + 2 x^3 + 5 x^2 - 2 x - 1)$\\ 
5&$(x - 1) (x^3 - 4 x^2 + 3 x + 1) (x^3 - 2 x^2 - x + 1) (x^6 - 6 x^5 + 8 x^4 + 8 x^3 - 13 x^2 - 6 x + 1)$ 
\end{tabular}

\

\begin{tabular}{c||c}
$n$&$\tau_n(x,x)$\\
\hline \hline
$-5$&$-2 (x - 1) (x^2 - 3 x + 1) (x^2 - 2 x - 1) (x^4 - 4 x^3 + 2 x^2 + 4 x - 1) (x^4 - 3 x^3 - x^2 + 3 x + 1)$\\
$-4$&$2 (x - 1) x (x^2 - 2 x - 2) (x^6 - 7 x^5 + 15 x^4 - 6 x^3 - 11 x^2 + 6 x + 1)$\\
$-3$&$-2 (x^2 - 3 x + 1) (x^2 - x - 1) (x^3 - 3 x^2 + 1)$\\
$-2$&$2 x (x^3 - 4x^2 + 3x + 1)$\\
$-1$&$-2(x-1)$\\ \hline
0&0\\
1&$2$\\ \hline
2&$-2 x (x^2 - 3 x + 1)$\\
3&$2 (x - 1) (x^2 - 2 x - 1) (x^3 - 3 x^2 + 1)$\\
4&$-2 (x - 1) x (x^2 - 2 x - 2) (x^5 - 6 x^4 + 10 x^3 - x^2 - 6 x + 1)$\\
5&$2 (x^2 - 3 x + 1) (x^3 - 4 x^2 + 3 x + 1) (x^3 - 2 x^2 - x + 1) (x^4 - 3 x^3 - x^2 + 3 x + 1)$\\
\end{tabular}

\

\begin{tabular}{c||c}
$n$&$k_n$\\ 
\hline  \hline
$-5$&$(x - 1) (x^2 - 2 x - 1) (x^4 - 4 x^3 + 2 x^2 + 4 x - 1)$\\ 
$-4$&$-x^6 + 7 x^5 - 15 x^4 + 6 x^3 + 11 x^2 - 6 x - 1$\\ 
$-3$&$-(x^2 - 3 x + 1) (x^2 - x - 1)$\\ 
$-2$&$x^3 - 4 x^2 + 3 x + 1$\\ 
$-1$&$x-1$\\ \hline 
0&$-1$\\
1&1\\ \hline 
2&$x^2 - 3 x + 1$\\
3&$-(x - 1) (x^2 - 2 x - 1)$\\
4&$-x^5 + 6 x^4 - 10 x^3 + x^2 + 6 x - 1$\\
5&$(x^3 - 4 x^2 + 3 x + 1) (x^3 - 2 x^2 - x + 1)$
\end{tabular} 

\

\begin{tabular}{c||c|c}
$n$&$g_n$&$h_n$\\
\hline  \hline
$-5$&$-(x^8 - 8 x^7 + 20 x^6 - 8 x^5 - 30 x^4 + 24 x^3 + 12 x^2 - 8 x - 1)$&$-(x - 1) (x^2 - 2 x - 1) (x^4 - 4 x^3 + 2 x^2 + 4 x - 1)$\\
$-4$&$-(x^6 - 5 x^5 + 5 x^4 + 6 x^3 - 7 x^2 - 2 x + 1)$&$- (x - 1) x (x^2 - 2 x - 2)$\\
$-3$&$(x - 1)  (x^4 - 4 x^3 + x^2 + 6 x + 1)$&$x^3 - 3 x^2 + 1$\\
$-2$&$x^3 - 2 x^2 - x + 1$&$x $\\
$-1$&$-(x^2-2x-1)$&$-1$\\ \hline
0&$-1$&0\\
1&$x-1$&1\\ \hline
2&$-(x^2 - x - 1)$&$-x$ \\
3&$-(x^4 - 4 x^3 + 2 x^2 + 4 x - 1)$&$- (x^3 - 3 x^2 + 1)$\\
4&$x^5 - 4 x^4 + 2 x^3 + 5 x^2 - 2 x - 1$&$ (x - 1) x (x^2 - 2 x - 2) $\\ 
5&$(x - 1)(x^6 - 6 x^5 + 8 x^4 + 8 x^3 - 13 x^2 - 6 x + 1)$ &$ (x^2 - 3 x + 1)  (x^4 - 3 x^3 - x^2 + 3 x + 1)$
\end{tabular}

\

\begin{tabular}{c||c}
$n$&$f_n(x,y)$\\
\hline  \hline
$-3$& $y^6-(3 x^2+1) y^5 + (3 x^4 + 8 x^2-5) y^4- (x^6 +13 x^4 -6 x^2-4) y^3$ \\
&$+ (6 x^6 + 11 x^4 - 24 x^2+6) y^2  - (12 x^6 - 16 x^4 - 2 x^2+3) y +(8 x^6 - 20 x^4 + 12 x^2 - 1)$\\ 
$-2$&$y^4 -(2 x^2 + 1) y^3+ (x^4 + 5 x^2 - 3) y^2 - (4 x^4 - x^2 - 2) y + (4 x^4  - 6 x^2 + 1)$\\ 
$-1$& $y^2 -(x^2 +1) y + (2 x^2 - 1)$\\ \hline 
0&1\\
1&$y-1$\\ \hline
2& $y^3 -(x^2+1)y^2  + (3 x^2-2) y  - (2 x^2  - 1)$\\ 
3&$ y^5-(2 x^2 + 1) y^4 + (x^4 + 6 x^2 - 4) y^3 -( 5 x^4-3) y^2 + (8 x^4 - 11 x^2 + 3) y  - (4 x^4 - 6 x^2 + 1)$  
\end{tabular}

\ 

\begin{tabular}{c||c}
$n$&$\tau_n(x,y)$\\
\hline  \hline
$-3$&$-2 (x^2 y - 2 x^2 - y^2 + 1) (x^3 y - 2 x^3 - x^2 y + 2 x^2 - x y^2 + 2 x + y^2 - 1)$\\
$-2$&$2 (x^3 y - 2 x^3 - x^2 y + 2 x^2 - x y^2 + x + y^2)$\\
$-1$&$-2 (x - 1)$\\ \hline
0&0\\
1&$2$\\ \hline
2&$-2 (x^2 y - 2 x^2 + x - y^2)$\\
3&$2 (x^2 y - 2 x^2 - y^2 + 1) (x^2 y - 2 x^2 + x - y^2 + 1)$
\end{tabular}  

}
\end{center}

We also attach the graphs of $f_n(x,y)=0$ and $\tau_n(x,y)=0$ in $\R^2$ for $n=-3,-2,-1,2,3$ drawn by Mathematica 12.0 \cite{Mathematica12}, where the blue curves are $f_n(x,y)=0$ and the orange curves are $\tau_n(x,y)=0$. 
 
\includegraphics[width=5cm]{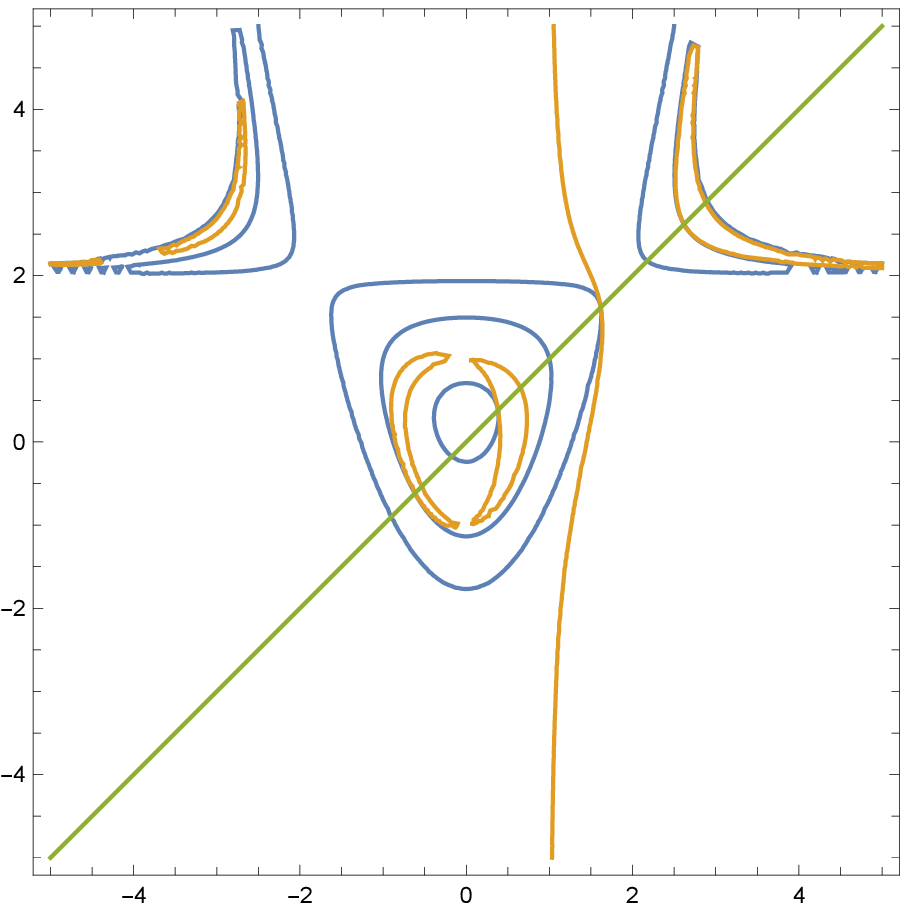} \ 
\includegraphics[width=5cm]{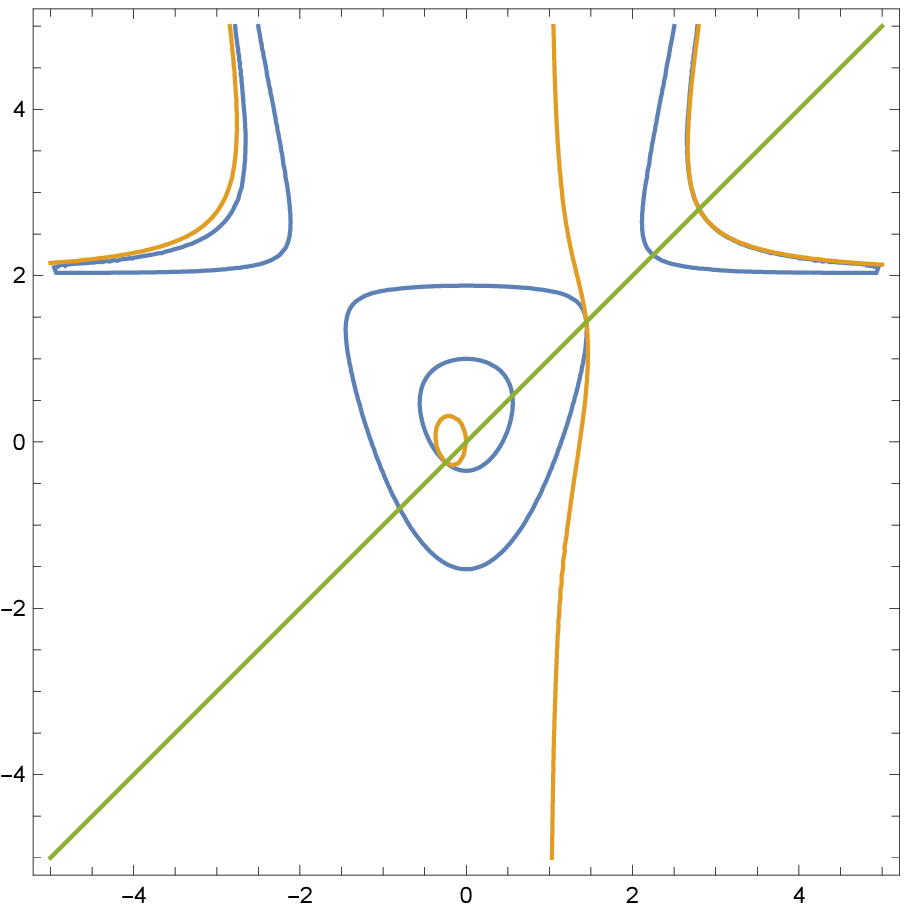} \ 
\includegraphics[width=5cm]{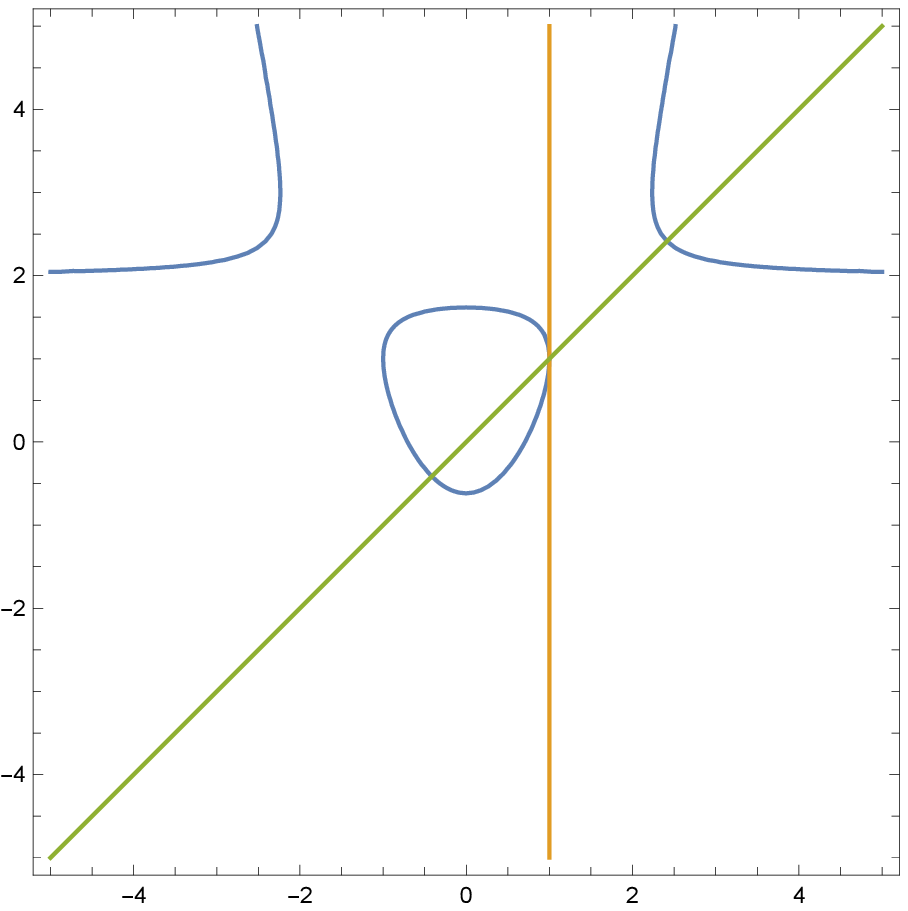} \ \\
\phantom{.} \hspace{23mm} $n=-3$ \hspace{37mm} $n=-2$ \hspace{37mm}  $n=-1$ 

\includegraphics[width=5cm]{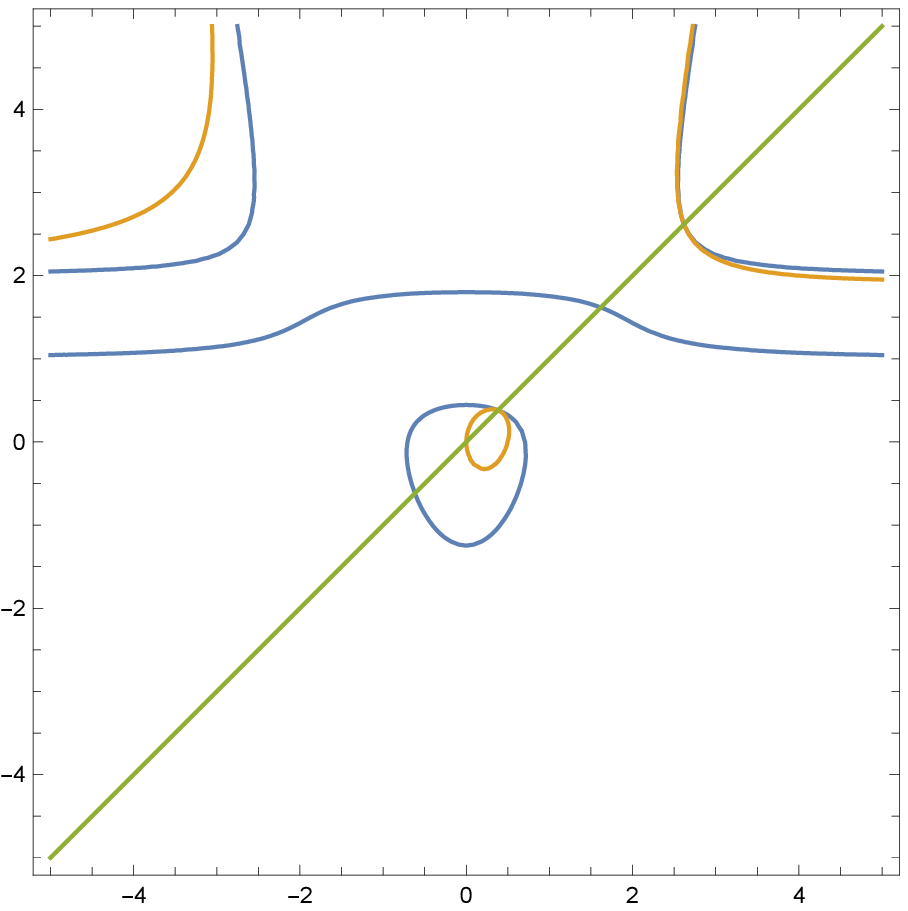} \ 
\includegraphics[width=5cm]{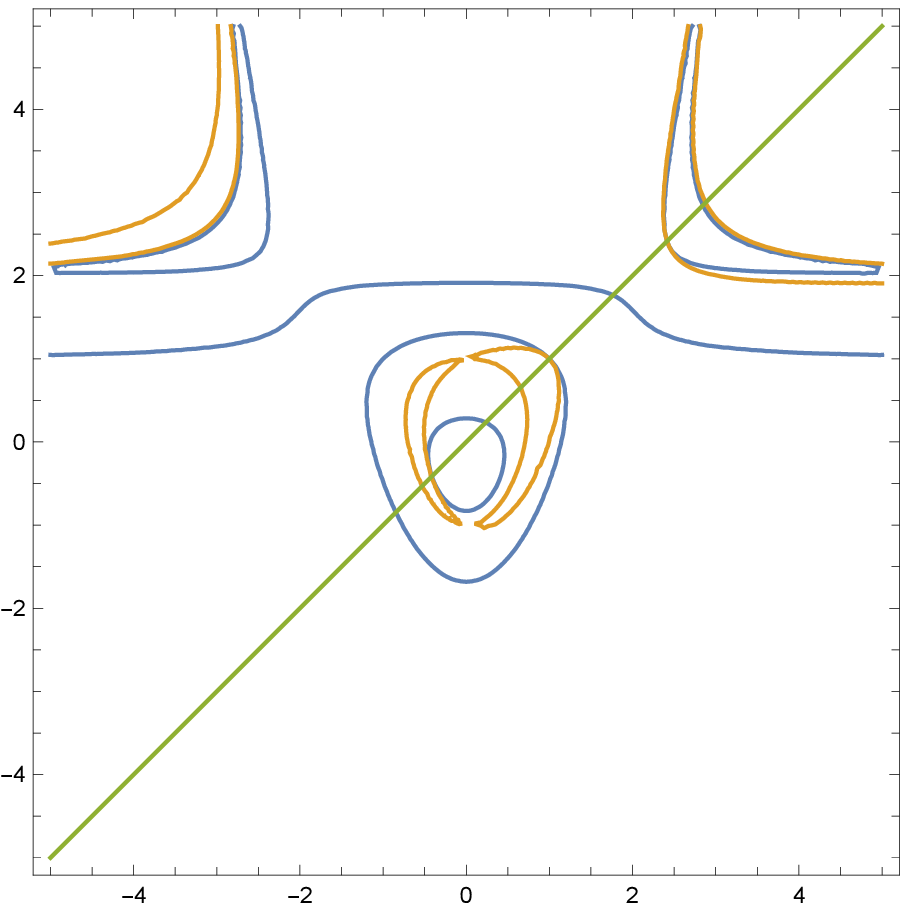} \\ 
\phantom{.} \hspace{23mm} $n=2$ \hspace{41mm} $n=3$  

\begin{rem} \label{rem.tangent} 
We may observe in the tables and graphs above that 
\begin{itemize} 
\item[(i)] $f_n(x,y)$ and $\tau_n(x,y)$ have a common root only on $x=y$, 
\item[(ii)] $g_n$ and $h_n$ have no common root, 
\item[(iii)] all roots of $f_n(x,x)$ and $\tau_n(x,x)$ are real numbers, 
\item[(iv)] $f_n(x,y)=0$ and $\tau_n(x,y)=0$ have a common tangent line at every intersection point.  
\end{itemize} 
\end{rem}

\section{Characterization} 
\subsection{$(-3)$-Dehn surgery} 
For a coprime pair of $q_1,q_2\in \Z$, let $M_{q_1/q_2}$ denote the result of the $q_1/q_2$-Dehn surgery along the twist knot $J(2,2n)$. Then we have 
$\pi_1(M_{q_1/q_2})=\langle a,b \,|\, aw^n=w^nb,\, a^{q_1}({\ol{w}^{n}}w^{n})^{q_2}=1\rangle$ for $\ol{w}=ba^{-1}b^{-1}a$, where 
the preferred longitude corresponding to a chosen meridian $\mu=a$ 
is $\lambda_a={\ol{w}^{n}}w^{n}$ 
by \cite[Section 2]{HosteShanahan2004JKTR}. 
(Note that the longitude in \cite{TranYamaguchi2018CMB} the inverse of this $\lambda_a$). 

Theorem \ref{thm.Dehn} precisely states the following: 
Let $\rho:\pi_n:=\pi_1(S^3-J(2,2n))\to \SL_2(\C)$ be an irreducible representation of the twist knot group. 
Then ${\rm tr}\rho(a)={\rm tr}\rho(ab)$ holds if and only if $\rho$ factors through the group $\pi_1(M_{-3})$ of $(-3)$-Dehn surgery, namely, there is a commutative diagram 
\[\xymatrix{
\pi_n \ar[r]^{\rho} \ar@{->>}[d]
& {{\rm SL}_2(\C)} \\ 
\pi_{1}(M_{-3}) \ar[ur] \ 
& \\
}\]
for the natural surjective homomorphism $\pi_n\surj \pi_1(M_{-3})$. 

The following remarks are due to several people: 
\begin{rem} \label{rem.Dehn} 
(1)  If we consider $J(-2,2n)$ instead, then the similar theorem holds for the $3$-Dehn surgeries. 

(2) The twist knot $J(2,2n)$ is hyperbolic iff $n\neq 0,1$. 
For each $n\neq 0,1$, 
the $(-3)$-Dehn surgery is an exceptional surgery, that is, $M_{-3}$ is not a hyperbolic manifold, and indeed is a small Seifert fibered space \cite[Theorem 1.1]{BrittenhamWu2001}. 
 
(3) Any irreducible metabelian $\SL_2(\C)$-representation of any knot group lies on ``$x=0$'' \cite[Proposition 1.1]{Nagasato2007KJM}, hence factors through the $4$-Dehn surgery. 
If an irreducible $\SL_2(\C)$-representation $\rho$ of a twist knot factors through the $4$-Dehn surgery, then $\rho$ is a metabelian representation \cite[Lemma 2.13, Proposition 3.1]{TranYamaguchi2018CMB}. 
\end{rem}

\subsection{Eigenvalues $M$ and $L$} 
In this subsection, we prove Theorem \ref{thm.Dehn}. 
Let $\rho:\pi_n \to \SL_2(\C)$ be an irreducible representation with $(\tr\rho(a),\tr\rho(ab))=(x,y)$. 
Then up to conjugate we may assume 
$$\rho(a)=\spmx{M&1\\0&M^{-1}}, \ \rho(b)=\spmx{M&0\\-u&M^{-1}}$$
with the matrix equality $\rho(aw^n)=\rho(w^nb)$, which is equivalent to $\Phi_n(x,u)=0$ for the Riley polynomial $\Phi_n(x,u)\in \Z[x,u]$. This $\rho$ is called Riley's representation. 
Here we have $x=M+M^{-1}$ and $-u=y-x^2+2$ for $x=\tr \rho(a)=M+M^{-1}$ and $y=\tr \rho(ab)$, hence $\Phi_n(x,u)=f_n(x,y)$.  

Since $\mu=a$ and $\lambda_a$ commute, we have $\rho(\lambda_a)=\spmx{L&*\\0&L^{-1}}$ for some $L$. 
The following cusp formula plays a key role in the proof of Theorem \ref{thm.Dehn}. 
\begin{prop}[{\cite[Equality (5.9)]{HosteShanahan2004JKTR}}] \label{prop.(5.9)}
We have $u=\dfrac{(1-M^2)(1-L)}{L+M^2}$. 
\end{prop}

\begin{lem} \label{lem.M+1} 
If $\rho$ factors through $\pi_1(M_{-3})$, then $M+1\neq 0$. 
\end{lem}
\begin{proof} Suppose $M=-1$. Then we have $\rho(\lambda_a)=\rho(a^3)=\spmx{-1&1\\0&-1}^{\!\!3}=\spmx{-1&3\\0&-1}$. 
On the other hand, \cite[Section 3.4.3]{DuboisHuynhYamaguchi2009} yields $\rho(\lambda_a)=\spmx{-1&(2u+4)/u\\0&-1}$, where $c=(2u+4)/u$ is called the cusp shape of $J(2,2n)$. Hence we have $(2u+4)/u=3$, $u=4$, $y=x^2-u+2=4-4+2=2$, and $z=x^2-x^2y+2y-2=4+8+8-2=18$. 
Now we may verify $f_0(-2,2)=f_1(-2,2)=1$ and that $f_n(-2,2)=f_{-n+1}(-2,2)$ for $n\geq1$ are increasing sequences, hence $f_n(-2,2)\neq 0$ for any $n\in \Z$. This implies that we have no such irreducible representation, hence contradiction. 
\end{proof} 

\begin{lem} \label{lem.M4neq1} 
If $\rho$ factors through $\pi_1(M_{-3})$ and satisfies $x=y$, then $M^4\neq 1$. 
\end{lem} 
\begin{proof} Recall $z(x,x)=-x^3+3x^2-2=-(x-1)(x^2-2x-2)$.

If $M=1$, then $x=2$, $z=2$, $f_0=f_1=1$, and $f_n=2$ for every $n\neq0,1$. 

If $M=-1$, then $x=-2$, $z=-18$, $f_0=1$, $f_1=-3$, and $f_{n+1}+18f_n+f_{n-1}=0$. 

If $M=\pm\sqrt{-1}$, then $x=\pm 2\sqrt{-1}$, $z=\mp8\sqrt{-1}+5$, 
$f_0=1$, $f_1=\pm\sqrt{-1}-1$, and $f_{n+1}+(\pm8\sqrt{-1}-5)f_n+f_{n-1}=0$. 

In each case it is easy to see $f_n\neq 0$, hence contradiction. 
\end{proof}

\begin{proof}[Proof of Theorem \ref{thm.Dehn}] 
If Riley's irreducible representation $\rho$ factors through $\pi_1(M_{-3})$, then we have $L=M^3$. 
We have $u(L+M^2)=uM^2(M+1)$ and $(1-M^2)(1-L)=(1+M)(1-M)(1-M^3)
=(1+M)M^2((M+M^{-1})^2-(M+M^{-1})-2)=(1+M)M^2(x^2-x-2)$.
Since $M+1\neq 0$ by Lemma \ref{lem.M+1}, Proposition \ref{prop.(5.9)} yields $u=x^2-x-2$. 
By $-u=y-x^2+2$, we obtain $y-x=0$. 

Conversely, suppose $x=y$. Proposition \ref{prop.(5.9)} yields $L=\dfrac{1-M^2-M^2u}{u+1-M^2} \cdots (\star)$. Note that $u+1-M^2\neq 0$, since otherwise we have $1-M^2-M^2u=u+1-M^2=0$, which yields $M^4=\pm1$, violating Lemma \ref{lem.M4neq1}. 
Since $u=x^2-x-2=M^2+M^{-2}-M-M^{-1}$, we have $1-M^2-M^2u=M^3(1-M-M^{-1}+M^{-2})$ and $u+1-M^2=(1-M-M^{-1}+M^{-2})$. Hence $(\star)$ yields $L=M^3$. Since $\rho(a)$ and $\rho(\lambda_a)$ commute, these two matrices are simultaneously diagonalizable. Therefore $L=M^3$ implies $\rho(a)^3=\rho(\lambda_a)$, 
hence $\rho$ factors through $\pi_1(M_{-3})$. 
\end{proof} 

\subsection{Doubly-indexed polynomials} \label{subsec.DIP}  
This subsection is optional. 
We introduce a doubly-indexed Chebyshev-like polynomials 
$\mf{d}_{m,n}\in \Z[x,y]$ and give an alternative proof for the only if part of Theorem \ref{thm.Dehn},
which is clearly workable over $\Z$, hence over a field with positive characteristic. 

Recall $\lambda_a=\ol{w}^n w^n$. 
Put $\mf{d}_{m,n}=\mf{d}_{m,n}(x,y)=\tr a^{-3}\ol{w}^mw^n \in \Z[x,y]$ for $m,n\in \Z$. 
Then we have 
$\mf{d}_{m+1,n}-z\mf{d}_{m,n}+\mf{d}_{m-1,n}=0$ and $\mf{d}_{m,n+1}-z\mf{d}_{m,n}+\mf{d}_{m,n-1}=0$ for every $m,n\in \Z$.

\begin{prop} \label{prop.b} 
We have 
$\mf{d}_{n,n}(x,x)-2=(x-2)(x+1)^2f_n(x,x)^2$ for every $n\in \Z$. 
\end{prop} 

\begin{proof} Assume $x=y$. A direct calculation of Riley's representation (or only of the trace relations) yields 
$\mf{d}_{0,0}-2=(x-2)(x+1)^2$ and $\mf{d}_{1,1}-2=(x-2)(x+1)^2(x-1)^2$. Put $\mf{D}_0=\mf{d}_{0,0}-2$. 
Then we in addition have $\mf{d}_{1,0}=\mf{d}_{0,1}=x(x-1)(x^2-x-1)$, hence an induction yields $\mf{d}_{n,n+1}=\mf{d}_{n+1,n}$ for every $n\in \Z$. 
Suppose that the equality $\mf{d}_{m,m}(x,x)-2=(x-2)(x+1)^2f_m(x,x)^2$ holds for $m=n-1, n$. 
By the Chebyshev recurrence formula, we have 
\begin{eqnarray*}
\mf{d}_{n+1,n+1}
&=&z\mf{d}_{n,n+1}-\mf{d}_{n-1,n+1}\\ 
&=&z(z\mf{d}_{n,n}-\mf{d}_{n,n-1})-(z\mf{d}_{n-1,n}-\mf{d}_{n-1,n-1})\\
&=&z^2\mf{d}_{n,n}-2z\mf{d}_{n,n-1}+\mf{d}_{n-1,n-1},
\end{eqnarray*} 
hence 
\begin{eqnarray*}
\mf{d}_{n+1,n+1}-2
&=&z^2(\mf{d}_{n,n}-2)-2z\mf{d}_{n,n-1}+(\mf{d}_{n-1,n-1}-2)+2z^2\\
&=&z^2\mf{D}_0f_n^2+\mf{D}_0f_{n-1}^2-2z(\mf{d}_{n,n-1}-z). \ \cdots \tc{1}
\end{eqnarray*} 
On the other hand, we have
\begin{eqnarray*}
\mf{D}_0f_{n+1}^2
&=&\mf{D}_0(zf_n-f_{n-1})^2\\
&=&z^2\mf{D}_0f_n^2+\mf{D}_0f_{n-1}^2-2z\mf{D}_0f_nf_{n-1}. \ \cdots \tc{2}
\end{eqnarray*} 
Now it suffices to show $\mf{D}_0f_nf_{n-1}=\mf{d}_{n,n-1}-z. \ \cdots \tc{3}$
We have $$\mf{d}_{1,0}-z=x(x-1)(x^2-x-1)-(-x^3+3x^2-2)=(x-1)(x+1)^2(x-2)=\mf{D}_0f_1f_0.$$ 
Again by induction, we have 
\begin{eqnarray*}
\mf{d}_{n,n-1}
&=&z\mf{d}_{n-1,n-1}-\mf{d}_{n-2,n-1}\\
&=&z(\mf{D}_0f_{n-1}^2+2)-(\mf{D}_0f_{n-1}f_{n-2}+z)\\
&=&\mf{D}_0(zf_{n-1}-f_{n-2})f_{n-1}+z\\
&=&\mf{D}_0f_nf_{n-1}+z,
\end{eqnarray*}
hence $\tc{3}$ holds. 
Thus by $\tc{1}$ and $\tc{2}$, we obtain $\mf{d}_{n+1,n+1}-2=\mf{D}_0f_{n+1}^2$. 

The similar argument also proves that the equalities for $m=n,n+1$ yields that for $m=n-1$. 
Thus we obtain the formula for every $n\in \Z$. 
\end{proof} 

\begin{proof}[Alternative partial proof of Theorem \ref{thm.Dehn}] 
If a representation $\rho$ is not on the line $x=2$, then both $\rho(a)$ and $\rho(\lambda_a)=\rho(\ol{w}^nw^n)$ have two distinct eigenvalues. 
Since $a$ and $\lambda_a=\ol{w}^nw^n$ commute, these two matrices are simultaneously diagonalizable. 
Hence $\rho(a^{-3}\lambda_a)=I_2$ if and only if $\tr \rho(a^{-3}\lambda_a)=2$, namely, $\mf{d}_n(x,y)=0$ at $(x,y)=(\tr\rho(a), \tr\rho(ab))$. 
Since each root of $f_n(x,x)=0$ corresponds to an irreducible $\SL_2(\C)$-representation on $x=y$ of $\pi_n$ by Proposition \ref{prop.f}, Proposition\ref{prop.b} yields the assertion. 
\end{proof} 

We remark that this assertion was initially proved by the first author by using Riley's explicit representation and the formula $A^n=\mca{S}_n(\tr A)A-\mca{S}_{n-1}(\tr A)I_2$ 
for any $A\in \SL_2(\C)$ and $n\in \Z$ \cite[Porposition 2.4]{Tran2016TJM}.

\subsection{A characterization} 
We further address the characterization of non-acyclic representations and prove Theorem \ref{thm.order3}, that is, 
\emph{irreducible $\SL_2(\C)$-representations of $J(2,2n)$ on $x=y$ are non-acyclic if and only if 
$\rho(a^{-1}w^n)$ is of order 3.}  
We suppose $x=y$ throughout this subsection. 

\begin{proof}[Proof of Theorem \ref{thm.order3}] 
Define a Chebyshev-like sequence by $c_n=\tr a^{-1}w^n \in \Z[x]$. 
Then we have $c_0=c_1=x$ and $c_n=x(\mca{S}_n(z)-\mca{S}_{n-1}(z))$. 
As in the proof of Proposition \ref{prop.roots}, 
by Lemmas \ref{lem.neq} and \ref{lem.1-x}, we have $z\neq \pm2$, $1-x=t+t^{-1}$ for some $t \in \C^*$, and $\mca{S}_n(z)=(t^{3n}-t^{-3n})/(t^3-t^{-3})$, so that we have 
$c_n=(1-t-t^{-1})((t^{3n}-t^{-3n})-(t^{3n-3}-t^{-3n+3}))/(t^3-t^{-3})
=-(t^{3n-1}+t^{-(3n-1)+1})(t+1)$. 
Recall that roots of $f_n(x,x)$ corresponds to $t$'s with $t^{3n-1}=\pm1$, while non-acyclic representations correspond to those with $t^{3n-1}=1$. 
If $t^{3n-1}=\pm1$, then $c_n=\mp1$. 

Suppose that $x$ is a root of $f_n(x,x)$. 
Then we have $\det(\rho(a^{-1}w^n)-I_2)\neq0$. Indeed, if $\det(\rho(a^{-1}w^n)-I_2)=0$, then by a general property of $\SL_2$, we have $\tr(\rho(a^{-1}w^n)-I_2)=0$, hence $c_n=2$, contradicting $c_n=\pm 1$. 
Note in addition that each $A
\in \SL_2$ satisfies $A^2-(\tr A)A+I_2=O$. 
Now we obtain the following equivalence: 

 $\rho(a^{-1}w^n)$ is of order 3 $\iffu$ $\rho(a^{-1}w^n)^3=I_2$ and $\rho(a^{-1}w^n)\neq I_2$
 
$\iffu$ $(\rho(a^{-1}w^n)-I_2)(\rho(a^{-1}w^n)^2+\rho(a^{-1}w^n)+I_2)=O$ and $\rho(a^{-1}w^n)\neq I_2$ 

$\iffu$ $\rho(a^{-1}w^n)^2+\rho(a^{-1}w^n)+I_2=O$ 
$\iffu$ $\tr \rho(a^{-1}w^n) +1=0$ $\iffu$  $c_n=-1$. 

Therefore, irreducible representations on $x=y$ with ${\rm ord}\rho(a^{-1}w^n)=3$ are exactly those which are non-acyclic.
\end{proof} 

\begin{rem} We indeed have $c_n+1=(x+1)k_nk_{-n+1}$. 
\end{rem} 

\begin{rem} \label{rem.Brieskorn} 
Sometimes the image of a non-acyclic representation of a twist knot is isomorphic to a von Dyck group, that is, the fundamental group of a Brieskorn manifold (cf.~\cite{Milnor1975Brieskorn, KitanoYamaguchi2016arXiv}). 
Since the image of representations are of interest also in the study of Galois deformation theory, 
Theorem D would give a new clue in arithmetic topology. 
\end{rem} 

\section{$L$-functions of universal deformations} \label{sec.L} 
In this section, we will give a complete answer to the Problem in the Introduction for twist knots. 

\subsection{Twisted invariants} \label{subsec.twisted} 
In this subsection, we overview the theory of Reidemeister torsions and twisted Alexander invariants, that will be used in our argument. A basic reference is \cite[Chapters 3, 6]{Hillman2}. 

\subsubsection{Reidemeister torsions} 
Let $R$ be any integral domain with the fraction field ${\rm fr}R$ and let $\dot{=}$ denote the equality up to multiplication by units in $R$. 
For a finitely generated free $R$-module $M$ with two bases $\mf{b}=(b_i)_i$ and $\mf{c}=(c_i)_i$, 
let $[\mf{b}/\mf{c}]=\det (p_{ij})_{ij}$ denote the determinant of the change of bases $b_i=\sum_j p_{ij}c_j$. 
Let $C_*=(0\to C_n \overset{\der_n}{\to}\cdots \overset{\der_2}{\to} C_1 \overset{\der_1}{\to} C_0\to 0)$ be a finitely generated free $R$-complex with a fixed basis $\mf{c}_*$ and suppose that $C_*$ is acyclic, that is, $H_*(C_*\otimes_R {\rm fr}R)=0$ holds. 
Let $\mf{b}_i$ be a basis of $\der_{i+1} C_{i+1}\subset C_i$ for each $i$ and let $\wt{\mf{b}}_{i-1}$ be a lift of $\mf{b}_{i-1}$ in $C_i$. Then the Reidemeister torsion of $(C_*,\mf{c}_*)$ is defined by 
$$\tau(C_*,\mf{c}_*)=\prod_i [\mf{b}_i\sqcup \wt{\mf{b}}_{i-1}/\mf{c}_i]^{(-1)^{i+1}}\in {\rm fr}R^*,$$
which is independent of the choice of $\mf{b}_*$. 
If instead $C_*$ is non-acyclic, then we put $\tau(C_*,\mf{c}_*)=0$. 
This torsion $\tau(C_*,\mf{c}_*)$ is multiplicative with respect to extensions \cite{Whitehead1950AJM, Milnor1966BAMS}: 
\begin{lem} \label{lem.tau-ext} 
Let $0\to A_* \to C_* \to B_* \to 0$ be an exact sequence of finitely generated acyclic free $R$-complexes with 
compatible bases $\mf{a}_*$, $\mf{c}_*$, $\mf{b}_*$, that is, we have $\mf{c}_*=\mf{a}_*\sqcup \wt{\mf{b}}_*$ for a lift $\wt{\mf{b}}_*$ of $\mf{b}_*$. 
Then we have $\tau(C_*,\mf{c}_*)=\tau(A_*,\mf{a}_*)\tau(B_*,\mf{b}_*)$. 
\end{lem} 

Suppose in addition that $R$ is a Noetherian UFD. 
The divisorial hull $\wt{\mf{a}}$ of an ideal $\mf{a}$ of $R$ is defined to be the intersection of all the principal ideals containing $\mf{a}$. 
For a finitely generated torsion $R$-module $M$, the order is defined to be a generator of the divisorial hull $\wt{{\rm Fitt}}_R(M)$ of the Fitting ideal. 
We have the following equality \cite{Turaev1986UMN}: 
\begin{lem} \label{lem.tau-ord} If $C_*$ is acyclic, then 
we have $\tau(C_*) $ $\dot{=}$ $ \prod_i {\rm ord}(H_i(C_*))^{(-1)^{i+1}}$. 
\end{lem} 

\subsubsection{Twisted invariants} 
Let $K$ be a knot in $S^3$ with the exterior $X=S^3-K$ and the group $\pi=\pi_1(X)$. 
Let $\pi=\langle a_1,\cdots, a_g\mid r_1,\cdots, r_{g-1} \rangle$ be a Wirtinger presentation and 
let $W$ denote the associated 
2-dimensional CW-complex with a natural basis $\mf{w}_*$.  
Note that $X$ is an Eilenberg-MacLane space $K(\pi,1)$, namely, $\pi_1(X)\cong \pi$ and $\pi_i(X)=1$ for $i>1$ holds,   and indeed $X$ is simply homotopic to $W$ by Waldhausen \cite{WaldhausenAM12}. 
The universal cover $\wt{W}\to W$ admits a natural left $\pi$-action and may be identified with the following finitely generated free $\Z[\pi]$-complex 
\[C_*=(0\to \Z[\pi]^{g-1} \overset{\der_2}{\to} \Z[\pi]^g \overset{\der_1}{\to} \Z[\pi]
\to 0)\]
called \emph{the Lyndon exact sequence} (cf.~\cite[Sections 4,5]{Lyndon1950Coh}), 
where each element is viewed as a column vector, 
the presentation matrix of $\der_2$ is the Jacobian matrix $P=(\dfrac{\der r_j}{\der a_i})_{i,j} \in \M_{g,g-1}(\Z[\pi])$ 
given by the Fox free derivatives, and that of $\der_1$ is $(a_j-1)_j\in \M_{1,g}(\Z[\pi])$. 
Let $\rho:\pi\to \GL_N(R)$ be a representation with $N\in \Z_{>0}$ and 
let $V_{\rho}$ denote the right representation space, namely, 
the set $V_{\rho}=R^N$ of column vectors with a right $\pi$-action via the transpose of $\rho$. 
Then the complex with local coefficients is defined by $C_{\rho,*}:=V_{\rho}\otimes_{\Z[\pi]}C_*$. 

The Reidemeister torsion of CW-complex is known to be an invariant of simple homotopy by Whitehead \cite{Whitehead1950AJM}, so that we only need to care about the choice of a basis. 
We choose a basis $\wt{\mf{w}}_*$ of $C_*$ which is a lift of $\mf{w}_*$ and a basis $\wt{\mf{w}}_{\rho,*}$ of $C_{\rho,*}$ consisting of the tensor products of elements of $\wt{\mf{w}}_*$ and the standard basis $\{\bm{e}_1,\cdots,\bm{e}_N\}$ of $V_{\rho}$. 
If we change the order of the basis, then $\tau_\rho(C_{\rho,*},\wt{\mf{w}}_{\rho,*})$ is multiplied by $\pm1$. 
If we change the choice of a lift $\wt{\mf{w}}$, then $\tau_\rho(C_{\rho,*},\wt{\mf{w}}_{\rho,*})$ is multiplied by $\det\rho(g)$ for some $g\in \pi$. 
Thus the Reidemeister torsion $\tau_\rho(X)=\tau(C_{\rho,*})$ of $(X,\rho)$ is defined up to multiplication by elements of $\pm \Im \det \rho$. 

Note that $V_{\rho}\otimes_{\Z[\pi]}\Z[\pi]\cong R^N$. 
Let $\rho$ also denote the linearly extended map 
$\rho:\M_{k,l}(\Z[\pi])\to \M_{k,l}(\M_N(R))\cong \M_{kN,lN}(R)$. 
Then the presentation matrix of $\der_2$ and $\der_1$ of $C_{\rho,*}$ are given by $\rho((\der r_j/\der a_i)_{i,j})\in \M_{g,g-1}(\M_N(R))$ and $\rho((x_j-1)_j)\in \M_{1,g}(\M_N(R))$ respectively. 
For each $1\leq k\leq g$, let $\vec{p}_k$ denote the $k$-th row of $P$, and $P_k$ the square matrix obtained from $P$ by deleting $\vec{p}_k$. 
Put $A_*=(0\to 0\to \Z[\pi]\overset{x_k-1}{\to} \Z[\pi]\to 0)$ and $C'_*=(0\to \Z[\pi]^{g-1}\overset{P_k}{\to} \Z[\pi]^{g-1}\to 0\to 0)$. 
By taking $V_\rho \otimes_{\Z[\pi]}$, we obtain an exact sequence 
\[0\to V_\rho \otimes_{\Z[\pi]} A_* \to C_{\rho,*} \to V_\rho \otimes_{\Z[\pi]} C'_* \to 0.\]
If $\det\rho(x_k-1)\neq 0$, then $V_\rho \otimes_{\Z[\pi]} A_*$ is acyclic, and so is $V_\rho \otimes_{\Z[\pi]} C'_*$. 
We may verify that $\tau (V_\rho \otimes_{\Z[\pi]} A_*)=\det \rho(x_k-1)^{-1}$ and $\tau(V_\rho \otimes_{\Z[\pi]} C'_*)=\det \rho(P_k)$, so that Lemma \ref{lem.tau-ext} yields the following explicit formula (cf. \cite{Johnson-note}). 
\begin{lem} \label{lem.tau-det} For each $1\leq k\leq g$, the equalities 
$\tau(C_{\rho,*},\wt{\mf{w}}_{\rho,*})=\dfrac{\det \rho(P_k)}{\det \rho(x_k-1)}$ in ${\rm fr}R^*$ and 
$\tau_\rho(X)=\dfrac{\det \rho(P_k)}{\det \rho(x_k-1)}$ in ${\rm fr}R^*/\pm \Im \det \rho$ hold. 
\end{lem} 

\begin{rem}
In our study of $\SL_2$-representations of twist knots, we start with the basis $\mf{w}$ given by a fixed Wirtinger presentation. We then choose a lift $\wt{\mf{w}}_*=((\wt{w}_{*,i})_i)_*$, 
which makes no ambiguity since $\det \SL_2=1$, and put 
$\wt{\mf{w}}_{\rho,*}:=(\wt{w}_{*,1}\otimes \bm{e}_1,$ $\wt{w}_{*,1}\otimes \bm{e}_2,$ $\wt{w}_{*,2}\otimes \bm{e}_1,$ $\wt{w}_{*,2}\otimes \bm{e}_2,$ $\cdots)_*$.  
Thus we may write $\tau_\rho=\tau(C_{\rho,*},\wt{\mf{w}}_{\rho,*})$ and we have $\tau_\rho=\dfrac{\det \rho(P_k)}{\det \rho(x_k-1)}$ in ${\rm fr}R^*$. 
Our convention essentially coincides with that in \cite{Tran2016TJM}. 
The relationship among several conventions is clarified in \cite[Section 2.8]{DunfieldFriedlJackson2012}. 

The Tieze argument also shows without using the simple homotopy theory that the ambiguity of $\tau_\rho$ due to the choice of basis is given by $\pm \Im \det \rho$ (cf. \cite{Wada1994}). 
\end{rem} 

The twisted homology is defined as that with local coefficient, so that we have an isomorphism $H_*(X,\rho)\cong H_*(C_{\rho,*})$. 
Since $X$ is an Eilenberg-MacLane space, the \emph{augmented} complex of $\wt{W}$ may be regarded as the following standard free resolution 
\[F_*=(0 \to \Z[\pi]^{g-1} \overset{\der_2}{\to} \Z[\pi]^g \overset{\der_1}{\to} \Z[\pi]\overset{\rm aug}{\to} \Z \to 0)\]
of $\Z$ over $\Z[\pi]$ given by the Fox free derivative (cf.~\cite[Proposition I.4.2, Exercise II.5.2(b)]{Brown})
and  
that of $C_{\rho,*}$ coincides with $V_\rho \otimes_{\Z[\pi]} F_*$, 
yielding an isomorphism $H_*(X,\rho)\cong H_*(\pi,\rho)$. 

\subsubsection{Alexander invariants}
For a representation $\rho:\pi\to \GL_N(R)$ and an abelianization map $\alpha:\pi\surj t^\Z \inj \GL_1(\Z[t^\Z])=\pm t^\Z$, 
the tensor representation $\rho\otimes \alpha:\pi\to \GL(R^n\otimes_\Z \Z[t^\Z])=\GL_n(R[t^\Z])$ is defined, 
extending to $\rho\otimes \alpha: \Z[\pi]\to M_n(R[t^\Z])$. 
By Lemma \ref{lem.tau-det}, the twisted invariant introduced by Lin and Wada in \cite{Lin2001AMS, Wada1994} 
coincides with the Reidemeister torsion $\tau_{\rho}(t)=\tau_{\rho\otimes \alpha}(X)$ and is well-defined up to multiplication by elements of $\pm \Im \det \rho\otimes \alpha=\pm t^\Z\Im \det \rho$ (cf. \cite{Kitano1996PJM}). 

Let $X_\infty\to X$ denote the $\Z$-cover corresponding to $\Ker \alpha$. 
Then Shapiro's lemma yields the natural isomorphisms $$H_i(X_\infty, V_\rho)\cong H_i(X,V_{\rho\otimes \alpha})= H_i(\pi,\rho\otimes \alpha)\cong H_i(\Ker \alpha, \rho)$$  
of $R[t^\Z]$-modules for each $i$. The orders $\Delta_{\rho,i}(t)\in R[t^\Z]/R[t^\Z]^*$ of these modules are called \emph{the $i$-th twisted Alexander polynomials}. 
By Lemma \ref{lem.tau-ord}, we have 
$$\tau_{\rho}(t) = \Delta_{\rho,1}(t)/\Delta_{\rho,0}(t) \ \text{in}\ {\rm fr}R[t^\Z]/R[t^\Z]^*$$
(cf. \cite[Theorem 4.1]{KirkLivingston1999T1}, \cite[Theorem 11]{HillmanLivingstonNaik2006}, \cite[Proposition 3.6]{SilverWilliams2009TA}). 

Since $\rho\otimes \alpha(g)|_{t=1}=\rho (g)$ for any $g \in \pi$, 
Lemma \ref{lem.tau-det} for $\tau_\rho(X)$ and $\tau_{\rho\otimes \alpha}(X)=\tau_\rho(t)$ yield that $\tau_\rho(1)=\tau_\rho(X)$. 
We have $\tau_{\rho}(X)$ $\dot{=}$ $\Delta_{\rho,1}(1)/\Delta_{\rho,0}(1)$ in ${\rm fr}R$, 
where we see that $\Delta_{\rho,i}(1)$ is the order of $H_i(X,\rho)$ for each $i$.

\subsection{Residual representations} \label{subsec.resrep} 

In this subsection, we study ${\rm SL}_2$-representations of twist knots over a finite field $F$ with characteristic $p>2$
and prove a version of Theorem \ref{thm.na}. 
In what follows, a \emph{residual representation} means an $\SL_2(F)$-representation, since we will consider its deformations later.  
Recall $\pi_n=\left<a,b \,|\, aw^n=w^nb\right>,$ $w=[a,b^{-1}]=ab^{-1}a^{-1}b.$
Since Le's argument in \cite{Le1993} and the proof of Proposition \ref{prop.f} are clearly applicable to the case over a field with positive characteristic, the character variety over $F$ has the same defining polynomial. Namely, we have the following assertion. 

\begin{prop}
Let $\ol{F}$ be an algebraic closure of $F$. 
Then the $\SL_2(\ol{F})$-character variety of $\pi_n$ is given by $(x^2-y-2)f_n(x,y)=0$. 
Each conjugacy class of irreducible representations corresponds to each point on $f_n(x,y)=0$ with $x^2-y-2\neq 0$. 
\end{prop}

A residual representation is said to be \emph{absolutely irreducible} if it is irreducible over an algebraic closure $\ol{F}$ of $F$. 
The following assertion is fundamental if we work over a finite field, and especially if we are interested in the rationality. 

\begin{prop} \label{prop.bijection} 
There is a bijection between the set ${\rm Rep}(\pi_n,F)_{\rm a.i.}/\!/\SL_2(F)$ of the conjugacy classes of absolutely irreducible representations $\ol{\rho}:\pi_n\to \SL_2(F)$ and the set ${\rm Ch}(\pi_n,F)_{\rm a.i.}$ of regular $F$-rational points $(\ol{\alpha},\ol{\beta})$ of $f_n(x,y)=0$ with $x^2-y-2\neq 0$. 
\end{prop} 

\begin{proof} If $F$ is replaced by an algebraic closure $\ol{F}$ of $F$, then the assertion is obtained by Riley's representation \cite[Lemma 7]{Riley1985}, as well as it follows from a general result \cite[Theorem 6.18]{Nakamoto2000PRIMS}. 
The Skolem--Noether theorem yields a natural injection ${\rm Rep}(\pi_n,F)_{\rm a.i.}/\!/\SL_2(F)\inj {\rm Rep}(\pi_n,\ol{F})_{\rm i.}/\!/\SL_2(\ol{F})$ into the set of conjugacy classes of irreducible representations over $\ol{F}$. 
A straight forward argument with use of the Hilbert Satz 90, that is, $H^1(F, \SL_2(\ol{F}))=1$ (cf.~\cite[Chapter X]{SerreGTM67}) yields the surjectivity of the composite 
$${\rm Rep}(\pi_n,F)_{\rm a.i.}/\!/\SL_2(F)\to {\rm Rep}(\pi_n,\ol{F})_{\rm i.}/\!/\SL_2(\ol{F})^{\Gal(\ol{F}/F)}\congto {\rm Ch}(\pi_n,F)_{\rm a.i.}$$ (cf.~\cite[Lemma 2.3.1]{Fukaya1998JA}, \cite[Proof of Theorem 1.1]{Harada2011PAMS}). 
We remark that the assertion indeed holds for a more general setting over any field with trivial Brauer group (cf.~\cite[Proposition 3.4]{Marche-RIMS2016}).  
\end{proof}

The calculation of Reidemeister torsion in \cite{Tran2016TJM} and Kitano's result in \cite{Kitano1996PJM} are also applicable to this case, so that we have the following assertion. 
\begin{prop} 
If $\ol{\rho}:\pi_n\to \SL_2(F)$ satisfies $\tr\ol{\rho}(a)=x, \tr\ol{\rho}(ab)=y$, then the Reidemeister torsion is given by 
$\tau_{\ol{\rho}}(S^3-J(2,2n))=\tau_n(x,y)$, 
where we regard that the value is zero if $\ol{\rho}$ is non-acyclic. 
\end{prop}

Each common zero of $f_n(x,y)$ and $\tau_n(x,y)$ with $x^2-y-2\neq 0$ corresponds to each absolutely irreducible non-acyclic residual representation. Recall that $k_n(x)$ is the greatest common divisor of $f_n(x,x)$ and $\tau_n(x,x)$ over $\Z$. 

\begin{prop} Let $-1 \neq \ol{\alpha}\in F$ be a root of $k_n$. 
Then $(\ol{\alpha},\ol{\alpha})$ is a regular $F$-rational point of $f_n(x,y)=0$. 
\end{prop} 

\begin{proof} The condition for $\ol{\rho}$ to be abelian is $\ol{\alpha}^2-2\ol{\alpha}-2=0$, that is, $\ol{\alpha}=-1,2$. 
By Corollary \ref{lem.k2}, we have $k_n(2)=\pm 1\neq 0$, hence we just assume $\ol{\alpha}\neq -1$. 

As calculated in Subsection \ref{subsec.tangent}, the partial derivatives at $(\ol{\alpha},\ol{\alpha})$ are given by 
\[\dfrac{\der f_n}{\der x}(\ol{\alpha},\ol{\alpha})=\dfrac{2((2n-1)\ol{\alpha}^2+(-4n+2)\ol{\alpha}+1)}{(\ol{\alpha}+1)\ol{\alpha}(\ol{\alpha}-2)(\ol{\alpha}-3)}, \ \ 
\dfrac{\der f_n}{\der y}(\ol{\alpha},\ol{\alpha})=\dfrac{2(n\ol{\alpha}^2-2n\ol{\alpha}-1)}{(\ol{\alpha}+1)\ol{\alpha}(\ol{\alpha}-2)(\ol{\alpha}-3)}\] 
and they indeed belong to $\Z[\ol{\alpha}]$. 

Suppose that $(\ol{\alpha},\ol{\alpha})$ is a singular point, so that we have $\dfrac{\der f_n}{\der x}(\ol{\alpha},\ol{\alpha})=\dfrac{\der f_n}{\der y}(\ol{\alpha},\ol{\alpha})=0$. 
Then we have 
$\left\{ \begin{array}{l} (2n-1)\ol{\alpha}^2+(-4n+2)\ol{\alpha}+1=0\\
n\ol{\alpha}^2-2n\ol{\alpha}-1=0\end{array} \right.$, 
hence $\left\{ \begin{array}{l} 3n-1=0 \\
(\ol{\alpha}-3)(\ol{\alpha}+1)=0\end{array} \right.$. 
However, by substituting $n=\dfrac{1}{3}\in F$, we obtain 
$\dfrac{\der f_n}{\der x}(\ol{\alpha},\ol{\alpha})=\dfrac{\der f_n}{\der y}(\ol{\alpha},\ol{\alpha})
=\dfrac{2}{3\ol{\alpha}(\ol{\alpha}-2)}$, which does not vanish if $\ol{\alpha}=-1,3$. 
Hence contradiction, and we obtain the assertion. 
\end{proof}

The following assertion is a version of Theorem \ref{thm.na}. 
\begin{prop} \label{prop.nap} 
Every absolutely irreducible non-acyclic residual representation corresponds to a root of $k_n$ in $F$. 
\end{prop} 

\begin{proof} 
Let $(x,y)$ be a common zero of $f_n(x,y)=0$ and $\tau_n(x,y)=0$ in $F^2$. 
Recall the argument of Lemma \ref{lem.neq}.
By a similar argument, we have $n\neq 0,1$ in $F$. 
If $z=2$, then we have $x^2-y-2=0$ and the point $(x,y)$ corresponds to an abelian representation. 
If $z=-2$ and $n$ is even, then we have $x=0$ and $z+2=\dfrac{1}{n^2}$, which is contradiction. 
If $z=-2$ and $n$ is odd, then we have $(x,y)=(\dfrac{2}{1-n},\dfrac{1}{n})$ and $z+2=\dfrac{(3n-1)^2}{n^2(n-1)^2}$,
which do not contradict if and only if $3n-1=0$, so that we have $x=y=3$ in $F$. 
In this case, by Corollary \ref{lem.k2}, $k_n(3)=\pm\dfrac{3n-1}{2}=0$ holds. 
The rest of the assertion with $z\neq \pm 2$ is obtained in a similar manner to the proof of Lemmas \ref{lem.neq}, \ref{lem.1-x} and Propositions \ref{prop.x=y}, \ref{prop.roots}. 
\end{proof}

Proposition \ref{prop.nap} yields that all common zeros of $f_n(x,y)$ and $\tau_n(x,y)$ on $x=y$ are still given by Proposition \ref{prop.roots}, and are exactly given as long as the values belong to $F$. 
Note that each $f\in F[[X]]$ may be uniquely written in the form $f=p^rX^su$, where $r\in \Q_{>0}$, $s\in \Z_{\geq 0}$, and $u$ is a unit in $F[[X]]$. 
The following lemma is obtained similarly to Corollary \ref{cor.tauk}. 
\begin{lem} \label{lem.tauk} 
Suppose $\dfrac{\der f_n}{\der y}(\ol{\alpha},\ol{\alpha})\neq 0$, so that 
Hensel's lemma for $f_n(x,y) \in F[[x-\ol{\alpha}]][y]$ yields the implicit function $\ol{y}_f(x)\in F[[x-\ol{\alpha}]]$ around $(\ol{\alpha},\ol{\alpha})$, and put $\ol{\tau}=\tau_n(x,\ol{y}_f(x))\in F[[x-\ol{\alpha}]]$. 
Then $\ol{\tau}=0$ holds only at $x=\ol{\alpha}$. Especially, $\ol{\tau}$ is not a constant. 

Suppose instead $\dfrac{\der f_n}{\der x}(\ol{\alpha},\ol{\alpha})\neq 0$, so that $\ol{x}_f(y)\in F[[y-\ol{\alpha}]]$ is defined and put $\ol{\tau}=\tau(\ol{x}_f(y),y) \in O[[y-\ol{\alpha}]]$. Then a similar result holds. 
\end{lem} 

\subsection{Universal deformations} \label{subsec.univdefo} 
In this subsection, we study universal deformations of absolutely irreducible residual representations of twist knots, 
whose definitions were given in the introduction. 
We fix a finite field $F$ with characteristic $p>2$ and a CDVR $O$ with the residue field $F$. 
For each local ring $R$, let $\mf{m}_R$ denote the maximal ideal. 
Representations $\rho,\rho'$ over a local ring $R$ are said to be \emph{strictly equivalent} if they are conjugate to each other and their images via $\mod \mf{m}_R$ coincide. 
A map satisfying the $\SL_2$-trace relations is called a \emph{pseudo-representation}, which was initially introduced by Wiles and Taylor for $\GL_2$ and $\GL_n$ (cf.~\cite[Section 1]{MTTU}, \cite[Subsection 1.1]{KMTT2018}). 
The following assertions due to Nyssen and Carayol play key rolls in our argument. Let $\pi$ be any group. 

\begin{lem}[{\rm \cite[Theorem 1]{Nyssen1996}}] \label{Nyssen} 
Let $R$ be a Henselian separated local ring with the maximal ideal $\mf{m}_R$ and the residue field $F$.
If there are a pseudo-representation $T:\pi\to R$ of degree two and 
an absolutely irreducible representation $\ol{\rho}:\pi\to \GL_2(F)$ with $\tr \ol{\rho}=T\mod \mf{m}_R$, 
then there is a representation $\rho:\pi\to \GL_2(R)$ with $\tr\rho=T$, which is unique up to strict equivalence. 
\end{lem}

\begin{lem}[{\rm \cite[Theorem 1]{Carayol1994}}] \label{Carayol} 
Let $\rho, \rho':\pi\to \GL_n(R)$ be representations over a local ring $R$ with the residue field $F$. 
If the residual representation $\rho\mod \mf{m}_R$ is absolutely irreducible and $\tr\rho=\tr\rho'$ holds, 
then $\rho$ and $\rho'$ are equivalent. 
\end{lem} 

The following lemma is a corollary of \cite[Theorem 2.2.4]{KMTT2018}. 

\begin{lem} \label{lem.ud} 
Let $\ol{\rho}:\pi_n\to \SL_2(F)$ be an absolutely irreducible residual representation. 

If $\alpha \in O$ is a lift of $\ol{\alpha}=\tr \ol{\rho}(a)$ 
and $\rho:\pi_n\to \SL_2(O[[x-\alpha]])$ is a deformation of $\ol{\rho}$ with $\tr \rho(a)=x$, 
then $\rho$ is a universal deformation of $\ol{\rho}$. 

Similarly, if $\beta\in O$ is a lift of $\ol{\beta}=\tr \ol{\rho}(ab)$ 
and $\rho:\pi_n\to \SL_2(O[[y-\beta]])$ is a deformation of $\ol{\rho}$ with $\tr \rho(ab)=y$, 
then $\rho$ is a universal deformation of $\ol{\rho}$. 
\end{lem} 

Universal deformations for twist knots are given by the following proposition. 

\begin{prop} \label{prop.univdefo} 
Let $(\alpha,\beta) \in O^2$ be a zero of $f_n(x,y)\in \Z[x,y]$ and suppose that the image $(\ol{\alpha},\ol{\beta}) \in F^2$ is a regular $F$-rational point of $f_n(x,y)=0$, that is, we have $\dfrac{\der f_n}{\der x}(\ol{\alpha},\ol{\beta})\neq 0$ or $\dfrac{\der f_n}{\der y}(\ol{\alpha},\ol{\beta})\neq 0$ in $F$. 
If $\dfrac{\der f_n}{\der y}(\ol{\alpha},\ol{\beta})\neq 0$, then the residual representation $\ol{\rho}$ at $(\ol{\alpha},\ol{\beta})$ admits a universal deformation $\bs{\rho}:\pi_n\to \SL_2(\mca{R})$ with $\mca{R}=O[[x-\alpha]]$. 
If instead $\dfrac{\der f_n}{\der x}(\ol{\alpha},\ol{\beta})\neq 0$, then so it does with $\mca{R}=O[[y-\beta]]$. 
\end{prop} 

\begin{proof}
If $\dfrac{\der f_n}{\der y}(\ol{\alpha},\ol{\beta})\neq 0$ in $F$, then we have $\dfrac{\der f_n}{\der y}(\alpha,\beta)\neq 0$ in $O$, and Hensel's lemma for $f_n(x,y) \in O[[x-\alpha]][y]$ 
yields a unique element $y_f(x) \in O[[x-a]]$ satisfying $y_f(\alpha)=\beta$ and $f(x,y_f(x))=0$. 
Define a map ${\bs T}:\pi_n \to O[[x-\alpha]]$ by ${\bs T}(a)=x,$ ${\bs T}(ab)=y_f(x)$, and the $\SL_2$-trace relations. 
(This map ${\bs T}$ may be directly proved to be the universal pseudo-representation of $\tr \ol{\rho}$ over $O[[x-\alpha]]$.) 
By Lemma \ref{Nyssen}, there exists a deformation $\bs{\rho}:\pi_n\to \SL_2(O[[x-\alpha]])$ over $O$ of $\ol{\rho}$, 
which is unique up to strict equivalence. 
Now Lemma \ref{lem.ud} yields that this $\bs{\rho}$ is a universal deformation over $O$ of $\ol{\rho}$. 

If instead $\dfrac{\der f_n}{\der x}(\ol{\alpha},\ol{\beta})\neq 0$, then we have $x_f(y)\in O[[y-\beta]]$ satisfying $x_f(\beta)=\alpha$ and $f(x_f(y),y))=0$, yielding the existence of a universal deformation $\bs \rho:\pi_n\to \SL_2(O[[y-\beta]])$. 
\end{proof} 

The proposition above slightly strengthen \cite[Corollaries 4.3, 4.4]{KMTT2018} by removing the assumption of the existence of a deformation over $O[[x-\alpha]]$.

Now we recall the relationship between universal deformations and Riley's universal representation. 
For each $(x,y)\in O^2$ with $f_n(x,y)=0$, Riley's representation at $(x,y)$ over a quadratic extension of $O$ is given by 
$$\rho^R(a)=\spmx{\dfrac{x+\sqrt{x^2-4}}{2}& 1\\0&\dfrac{x-\sqrt{x^2-4}}{2}}, \ \ \rho^R(b)=\spmx{\dfrac{x+\sqrt{x^2-4}}{2}& 0\\ -(x^2-y-2)&\dfrac{x-\sqrt{x^2-4}}{2}}$$
(cf. \cite[Lemma 7]{Riley1985}). 
This $\rho^R$ may be 
regarded a function over the curve $f_n(x,y)=0$ and is called Riley's universal representation. 

Suppose again $(\alpha,\beta)\in O^2$ with $f_n(\alpha,\beta)=0$ such that a residual representation $\ol{\rho}$ at the image $(\ol{\alpha},\ol{\beta}) \in F^2$ is absolutely irreducible. 
If $\dfrac{\der f_n}{\der y}(\ol{\alpha},\ol{\beta})\neq 0$, then by putting $y=y_f(x)\in O[[x-\alpha]]$, this $\rho^R$ may be regarded as that over a quadratic extension of $O[[x-\alpha]]$. 
Since $\tr \rho^R$ is a deformation of $\tr\ol{\rho}$ over $O[[x-\alpha]]$, by Theorems of Nyssen and Carayol, 
this $\rho^R$ is equivalent to some deformation ${\bs \rho}$ of $\ol{\rho}$ over $O[[x-\alpha]]$, 
which is indeed a universal deformation of $\ol{\rho}$ over $O$ by Lemma \ref{lem.ud}. 

Now recall that the Reidemeister torsion is an invariant of conjugacy classes of representations. 
Since the explicit calculation of Reidemeister torsion with use of the Fox free derivative is compatible with putting $y=y_f(x)$, 
we see that $\tau_{\bs \rho}$ is the image of $\tau_n(x,y)$ via 
the natural map $O[x,y]\to O[[x-\alpha]]; y\mapsto y_f(x)$ given by Hensel's lemma. 

If instead $\dfrac{\der f_n}{\der x}(\alpha,\beta)\neq 0$, then we similarly obtain a universal deformation ${\bs \rho}$ over $O[[y-\beta]]$ so that 
$\tau_{\bs \rho}$ is the image of $\tau_n(x,y)$ via $O[x,y]\to O[[y-\beta]]; x\mapsto x_f(y)$. 

\begin{prop} \label{prop.tau} 
Let the notation be as in Proposition \ref{prop.univdefo}. 
If $y_f(x)$ is defined, then we have $\tau_{\bs \rho}=\tau_n(x,y_f(x)) \in O[[x-\alpha]]$. 
If $x_f(y)$ is defined, then we have $\tau_{\bs \rho}=\tau_n(x_f(y),y) \in O[[y-\beta]]$. 
\end{prop}

\subsection{Non-trivial $L$-functions of twist knots} \label{subsec.L} 

In this subsection, we determine all residual representations with non-trivial $L$-functions, as well as determine the $L$-function themselves. Let $F$ and $O$ be as in the previous section and 
let $\ol{\rho}:\pi_n\to \SL_2(F)$ be an absolutely irreducible residual representation 
admitting a universal deformation ${\bs \rho}:\pi_n\to \SL_2(\mca{R})$ over $O$. 
\begin{prop} The 1st homology group $H_1(\pi_n, {\bs \rho})$ is a finitely generated torsion $\mca{R}$-module, hence the $L$-function $L_{\bs \rho}$ is defined. 
\end{prop} 
\begin{proof} Suppose that $\ol{\rho}$ corresponds to a point $(\ol{\alpha},\ol{\beta}) \in F^2$. 
If $\dfrac{\der f_n}{\der y}(\ol{\alpha},\ol{\beta})\neq 0$, then we have a universal deformation ${\bs \rho}:\pi_n\to \SL_2(O[[x-\alpha]])$,
where $O[[x-\alpha]]$ is a Noetherian integral domain. 
In addition, for any deformation $\alpha \in O$ of $\ol{\alpha}$, 
there exists a unique deformation $\beta \in O$ of $\ol{\beta}$ satisfying $f_n(\alpha,\beta)=0$. 
By Theorem \ref{thm.na} over an algebraic closure of ${\rm fr}O$, 
for all but finitely many such $\alpha$, we have $\tau_n(\alpha,\beta)\neq 0$, that is,
a deformation $\rho:\pi_n\to \SL_2(O)$ at $(\alpha,\beta)$ is acyclic and $\Delta_{\rho,1}(1)\neq 0$ holds. 
Hence we may take $\alpha$ such that $\rho$ is acyclic. 
We may further assume $\det(\rho(a)-I_2)=2-\tr \rho(a)=2-\alpha \neq 0$. 
Now \cite[Theorem 3.2.4]{KMTT2018} yields that $H_1({\bs \rho})$ is a finitely generated torsion $O[[x-\alpha]]$-module. 
If instead $\dfrac{\der f_n}{\der x}(\ol{\alpha},\ol{\beta})\neq 0$, then a similar argument holds. Hence we obtain the assertion. 
\end{proof}

Note that in the proof above we used the existence of an acyclic deformation of $\ol{\rho}$. 
To the contrary, we will later use a non-acyclic deformation in the calculation of the $L$-function. 

As we remarked in the introduction, the $L$-function $L_{\bs \rho}$ is defined to be the order $\Delta_{\bs \rho,1}(1)$ of $H_1(\pi_n, {\bs \rho})$. 
\begin{rem} \label{rem.L} 
These $L$-functions $L_{\bs \rho}=\Delta_{\bs \rho,1}(1)$ may be seen as infinitesimal ${\rm SL}_2$-analogues of the classical Alexander polynomials. 
Perhaps one may think that the $L$-function should be defined to be $\Delta_{\bs \rho,1}(t)$, instead of $\Delta_{\bs \rho,1}(1)$. However, since the Alexander parameter $t$ may be seen as that of universal deformation of 1-dimensional trivial representation, it is natural only to consider the $\SL_2$-deformation parameters. 
\end{rem} 
The description at the end of Subsection \ref{subsec.twisted} yields the following equality. 
\begin{prop} \label{prop.L} 
We have $L_{\bs \rho}$ $\dot{=}$ $ \tau_{\bs \rho} \Delta_{{\bs \rho},0}(1)$ in $\mca{R}$. 
\end{prop}

Noting in addition that we have $\tau_{\ol{\rho}} = \tau_{\bs \rho} \mod \mf{m}_\mca{R}$ in $F$, 
Lemma \ref{lem.tauk} yields the following assertion. 
\begin{prop} \label{prop.L1} 
The following conditions are equivalent: 
{\rm (i)} $\ol{\rho}$ is acyclic, {\rm (}that is, $\tau_{\rho}\neq 0$,{\rm )} 
{\rm (ii)} $\tau_{\bs \rho} $ $\dot{=}$ $ 1$, 
{\rm (iii)} $L_{\bs \rho} $ $\dot{=}$ $ \Delta_{\bs{\rho}, 0}(1)$. 
If $L_{\bs \rho} $ $\dot{=}$ $ 1$, then $\ol{\rho}$ is acyclic. 
\end{prop} 

Here is a basic tool to investigate the $L$-functions in $O[[X]]=O[[x-\alpha]]$ or $O[[y-\alpha]]$. %
\begin{lem}[{$p$-adic Weierstrass preparation theorem, \cite[Theorem 5.3.4]{NSW}}] \label{lem.Weierstrass} 
Each $f(X) \in O[[X]]$ may be uniquely written in the form $f(X)=p^r g(X)u(X)$, 
where $r \in \Q_{>0}$, $u(X) \in O[[X]]$ is a unit in $O[[X]]$, and $g(X)=X^s+(\text{lower\ terms}) \in O[X]$ is a \emph{Weierstrass polynomial}, that is, a polynomial satisfying $g(X)-X^s\in \mf{m}_O[X]$. 
\end{lem} 

\begin{prop} \label{prop.L=tau} 
If $\ol{\rho}$ is non-acyclic, then we have $\Delta_{{\bs \rho},0}(1) $ $\dot{=}$ $ 1$, hence $L_{\bs \rho}$ $\dot{=}$ $ \tau_\bs{\rho} $ $\dot{\neq}$ $ 1$ holds. 
\end{prop} 

\begin{proof} 
We have $\pi_n=\langle a,b\mid r\rangle$ with $r=aw^nb^{-1}w^{-n}$. 
Put $V_\rho=\mca{R}^2$ and consider the right $\pi_n$-action via the transpose of ${\bs \rho}$. 
The Lyndon exact sequence gives the twisted complex 
$0\to V\overset{\der_2}{\to} V^2 \overset{\der_1}{\to} V\to 0$ with 
$\der_2=(\bs{\rho}(\dfrac{\der r}{\der a}), \bs{\rho}(\dfrac{\der r}{\der b}))$ 
and $\der_1=(\bs{\rho}(a)-I_2, \bs{\rho}(b)-I_2)$ defining $H_i(\pi_n, {\bs \rho})$. 
Note that $\Delta_{{\bs \rho},0}(1)$ is the GCD of the 2-minors of $\der_1$, so that it divides $2-\tr \bs{\rho}(a)=(2-\alpha)-(\tr \bs{\rho}(a)-\alpha)$. 
In the residual field $F$, we have $k_n(\ol{\alpha})=0$, while Corollary\ref{lem.k2} yields $k_n(2)=\pm 1$. 
Hence $\ol{\alpha}\neq 2$, $|2-\alpha|_p\geq 1$ (indeed $=1$), and $2-\tr \bs{\rho}(a)$ is a unit in $\mca{R}$, 
hence so is $\Delta_{\bs{\rho},0}(1)$. 
Now Propositions \ref{prop.L}, \ref{prop.L1} yield $L_{\bs \rho} $ $\dot{=}$ $ \tau_\bs{\rho} $ $\dot{\neq}$ $ 1$. 
\end{proof} 

\begin{prop} \label{prop.L=x-2} 
Suppose that $\ol{\rho}$ is acyclic. Then $L_{\bs \rho}$ is non-trivial if and only if $\tr\ol{\rho}(a)=2$ in $F$ holds.  
In this case, if $\mca{R}=O[[x-\alpha]]$, then we have $L_{\bs \rho}$ $\dot{=}$ $x-2$. 
If instead $\mca{R}=O[[y-\beta]]$, then we have $L_{\bs \rho} $ $\dot{=}$ $ x_f(y)-2$ or $L_{\bs \rho} $ $\dot{=}$ $ \sqrt{x_f(y)-2}$, where the latter holds if and only if we have $\sqrt{x_f(y)-2} \in O[[y-\beta]]$. 
\end{prop} 

\begin{proof} 
Suppose $\mca{R}=O[[x-\alpha]]$. 
Recall that $L_{\bs \rho} $ $\dot{=}$ $ \Delta_{{\bs \rho},0}(1)$ and 
that $\Delta_{{\bs \rho},0}(1)$ is the GCD of the 2-minors of $({\bs \rho}(a)-I_2, {\bs \rho}(b)-I_2)$, 
which is calculated to be that of $2-x$ and $(\dfrac{x+\sqrt{x^2-4}}{2}-1)(-u-1)=-\sqrt{x-2}\dfrac{1+\sqrt{x+2}}{2}(u+1)$ in $O[[x-\alpha]]$. Since $\sqrt{x-2}\not\in F[[x-2]]$, we have $\sqrt{x-2}\not\in O[[x-\alpha]]$, hence $L_{\bs \rho}$ $\dot{=}$ $x-2$ in $O[[x-\alpha]]$. 
We have $x-2 $ $\dot{\neq}$ $ 1$ if and only if $\tr \ol{\rho}(a)=2$ in $F$ holds. 
If instead $\mca{R}=O[[y-\beta]]$, then by regarding $x=x_f(y)$, a similar argument yields the assertion. 
\end{proof} 

\begin{rem} 
A lift $\rho_{\rm hol}^{\pm}:\pi_n\to \SL_2(\C)$ of the holonomy representation may be regarded as that over the ring of $S$-integers of a number field. As proved by W.~P.~Thurston \cite{ThurstonGT}, such $\rho_{\rm hol}^{\pm}$ is parabolic, that is, we have $\tr\rho(a)=2$. 
We wonder if we always have $L_{\bs \rho}=x-2$ or $y-\beta$ for the residual representations of $\rho_{\rm hol}^{\pm}$ (cf.~Example \ref{eg.L=x-2}). 
\end{rem} 

By Propositions \ref{prop.L1}, \ref{prop.L=tau}, \ref{prop.L=x-2}, all residual representations with non-trivial $L$-functions are determined. 
The following lemma connects Theorem \ref{thm.tangent} (the common tangent property) to the property of the $L$-functions. 

\begin{lem} \label{lem.order2} 
Let $\alpha\in O$ be a root of $k_n$ with the image $\ol{\alpha}$ in $F$ and 
let $\ol{\rho}$ be a residual representation at $(\ol{\alpha},\ol{\alpha})$. 
Suppose $\dfrac{\der f_n}{\der y}(\ol{\alpha},\ol{\alpha})\neq 0$, so that 
Hensel's lemma yields the implicit function $y_f(x)\in O[[x-\alpha]]$ and 
there is a universal deformation ${\rm \rho}:\pi_n\to \SL_2(O[[x-\alpha]])$. 
Then we have $L_{\bs \rho}=(x-\alpha)^2\times (\text{other\ divisors})$.
In addition, $L_{\bs \rho}$ is monic, that is, $L_{\bs \rho} $ $\dot{=}$ $ g$ holds for some Weierstrass polynomial $g$. 

Suppose instead $\dfrac{\der f_n}{\der y}(\ol{\alpha},\ol{\alpha})\neq 0$. Then the similar holds for $L_{\bs \rho} \in O[[y-\alpha]]$. 
\end{lem}  

\begin{proof} Suppose $\dfrac{\der f_n}{\der y}(\ol{\alpha},\ol{\alpha})\neq 0$. 
By Propositions \ref{prop.nap}, \ref{prop.tau}, \ref{prop.L=tau}, 
$\ol{\rho}$ is non-acyclic and we have $L_{\bs \rho}=\tau_n(x,y_f(x))$. 
Since the partial derivatives of $f_n(x,y)$ and $\tau_n(x,y)$ are polynomials in $\Z[x,y]$, 
the calculations in Subsection \ref{subsec.tangent} and especially Corollary \ref{cor.d2} still hold 
for $y_f(x) \in O[[x-\alpha]]$. 
Namely, if $\dfrac{\der f_n}{\der y}(\ol{\alpha},\ol{\alpha})=\dfrac{2(n\ol{\alpha}^2-2n\ol{\alpha}-1)}{(\ol{\alpha}+1)\ol{\alpha}(\ol{\alpha}-2)(\ol{\alpha}-3)}\neq 0$, then $L_{\bs \rho}\in O[[x-\alpha]]$ satisfies 
$$L_{\bs \rho}|_{x=\alpha}=\dfrac{d L_{\bs \rho}}{dx}|_{x=\alpha}=0, 
\dfrac{d^2L_{\bs \rho}}{dx^2}|_{x=\alpha}=\dfrac{2(3n-1)^2}{(\alpha-2)(n\alpha^2-2n\alpha-1)^2}\neq 0.$$ 

Now recall Lemma \ref{lem.tauk}. Since $\ol{y}_f(x) \in F[[x-\alpha]]$ is the image of $y_f(x)\in O[[x-\alpha]]$, 
$\ol{\tau}=\tau_n(x,\ol{y}_f(x)) \in F[[x-\alpha]]$ is the image of $L_{\bs \rho}=\tau_n(x,y_f(x))$. 
Since $\ol{\tau}\neq 0$, $L_{\bs \rho}$ is not divisible by $p$, that is, $L_{\bs \rho}$ is monic in $O[[x-\alpha]]$.

If instead $\dfrac{\der f_n}{\der x}(\ol{\alpha},\ol{\alpha})=\dfrac{2((2n-1)\ol{\alpha}^2+(-4n+2)\ol{\alpha}+1)}{(\ol{\alpha}+1)\ol{\alpha}(\ol{\alpha}-2)(\ol{\alpha}-3)}\neq 0$, then $L_{\bs \rho}\in O[[y-\alpha]]$ satisfies 
$$L_{\bs \rho}|_{y=\alpha}=\dfrac{d L_{\bs \rho}}{dy}|_{y=\alpha}=0, 
\dfrac{d^2L_{\bs \rho}}{dy^2}|_{y=\alpha}=\dfrac{2(3n-1)^2}{(\alpha-2)((2n-1)\alpha^2+(-4n+2)\alpha+1)^2}\neq 0.$$
A similar argument yields that $L_{\bs \rho}$ is monic in $O[[y-\alpha]]$. 
\end{proof} 

The following theorem is a full version of Theorem E, determining the $L$-function of all non-acyclic $\ol{\rho}$'s. 

\begin{thm} \label{thm.L.precise}
Let $F$ be a finite field with characteristic $p>2$ and let $O$ be a CDVR with the residue field $F$. 
Let $\ol{\rho}:\pi_n\to\SL_2(F)$ be an absolutely irreducible non-acyclic representation 
corresponding to a root $\ol{\alpha}$ of $k_n$ in $F$ 
and suppose that $O$ contains a root $\alpha$ of $k_n$ such that $\alpha$ is a lift of $\ol{\alpha}$. 

If $\dfrac{\der f_n}{\der y}(\ol{\alpha},\ol{\alpha})\neq 0$, so that $\ol{\rho}$ admits a universal deformation $\bs \rho: \pi_n \to \SL_2(O[[x-\alpha]])$ over $O$, 
then the $L$-function satisfies $L_{\bs \rho} $ $\dot{=}$ $ k_n(x)^2$ in $O[[x-\alpha]]$. 
If $p\nmid 3n-1$, then $\ol{\alpha}$ is always a single root of $k_n$ and $L_{\bs \rho} $ $\dot{=}$ $ (x-\alpha)^2$ holds. 
If $p\mid 3n-1$, then the multiplicity of $\ol{\alpha}$ may be $m>1$, so that there are exactly $m$ distinct lifts $\alpha_1=\alpha,\cdots,\alpha_m$ of $\ol{\alpha}$ that are roots of $k_n$ in an extension of $O$, 
and $L_{\bs \rho} $ $\dot{=}$ $ \prod_{1\leq i\leq m} (x-\alpha_i)^2$ holds. 

If instead $\dfrac{\der f_n}{\der x}(\ol{\alpha},\ol{\alpha})\neq 0$ so that 
$\ol{\rho}$ admits a universal deformation $\bs \rho: \pi_n \to \SL_2(O[[y-\alpha]])$, 
then the $L$-function satisfies $L_{\bs \rho} $ $\dot{=}$ $ k_n(y)^2$ in $O[[y-\alpha]]$, and the similar formulas hold. 
If both of $\dfrac{\der f_n}{\der x}(\ol{\alpha},\ol{\alpha})\neq 0$ and $\dfrac{\der f_n}{\der y}(\ol{\alpha},\ol{\alpha})\neq 0$ hold, 
then the $L$-functions in $O[[x-\alpha]]$ and $O[[y-\beta]]$ coincide via $x=y$. 
\end{thm} 

\begin{proof} 
If $\ol{\alpha}$ is a root of $k_n$ in $F$ with multiplicity $m$, then by taking a larger $O$ if necessary, 
we may assume that it lifts to distinct $m$ roots $\alpha_1=\alpha,\cdots, \alpha_m$ of $k_n$ in $O$.  
Suppose $\dfrac{\der f_n}{\der y}(\ol{\alpha},\ol{\alpha})\neq 0$, so that Hensel's lemma yields the implicit function $y_{f,i}(x)\in O[[x-\alpha_i]]$ around each lift $(\alpha_i,\alpha_i)$. 
We have $|(x-\alpha_i)-(x-\alpha)|_p<1$, hence $O[[x-\alpha]]=O[[x-\alpha_i]]$ for each $i$. 
A standard argument of Newton's method (cf.~\cite[Theorem 7.3]{EisenbudGTM}) 
yields that $y_{f,i}(x)$'s are the same elements in $O[[x-\alpha]]$. 
By Lemma \ref{lem.Weierstrass}, $L_{\bs \rho}=\tau_n(x,y_f(x)) \in[[x-\alpha]]$ vanishes only at roots $\beta$ of $k_n$ with $|\beta-\alpha|_p<1$, that is, $\beta=\alpha_1,\cdots, \alpha_m$. 
By Lemma \ref{lem.order2}, we obtain $L_{\bs \rho} $ $\dot{=}$ $ \prod_i (x-\alpha_i)^2 $ $\dot{=}$ $ k_n(x)^2$. 
The case with $\dfrac{\der f_n}{\der x}(\ol{\alpha},\ol{\alpha})\neq 0$ may be treated similarly. 
\end{proof} 

\begin{rem} \label{rem.final} (1) 
As mentioned in the introduction, there is a remarkable analogy between Thurston's hyperbolic deformation theory and Hida--Mazur's Galois deformation theory. 
Although non-acyclic representations factor through exceptional Dehn surgeries (Remark \ref{rem.Dehn} (2)), our representations are still in an analogy with $p$-ordinary representations in the sense of \cite{MorishitaTerashima2007}. 
Indeed, since the peripheral subgroup $D$ of a knot group is abelian, the restriction $\rho|_D$ of a knot group representation is always equivalent to $\spmx{\chi_\rho &\ast\\ 0&\chi_\rho^{-1}}$ for some character $\chi_\rho$. 
Thus our work lives in the complement of the previous scope in an extended whole picture, that would be described in future. 

(2) 
For a lift of the holonomy representation, we have a geometric analogue of Iwasawa main conjecture $R_X(z,\rho)=A_\rho(z)^2$ in $\C[[z]]$ due to Sugiyama \cite{Sugiyama2007, Sugiyama2009ASPM}, where $R_X(z,\rho)$ is the Ruelle--Selberg zeta function of geodesics in the analytic side and $A_\rho(z)$ is the Alexander invariant in the algebraic side. 
His argument is based on Fried's work \cite{Fried1986Invent}.  
We wonder how it would be like if we establish its analogue in our situation of non-acyclic representations with use of our algebraic $L$-function. 
\end{rem}

\subsection{Examples} \label{subsec.eg} 
We exhibit examples for Propositions \ref{prop.bijection}, \ref{prop.L=x-2}, and Theorem \ref{thm.L.precise}. 

\begin{eg} 
Let $x,y \in \F_p$ with $f_n(x,y)=0$ and put $u=x^2-y-2$, $v=\sqrt{1-\dfrac{x^2-4}{u}}$. 
Then Riley's representation of $\pi_n$ over a quadratic extension $\F_{p^2}$ of $\F_p$ is given by 
$$\rho^R(a)=\spmx{\dfrac{x+\sqrt{x^2-4}}{2}& 1\\0&\dfrac{x-\sqrt{x^2-4}}{2}}, \ \ \rho^R(b)=\spmx{\dfrac{x+\sqrt{x^2-4}}{2}& 0\\ -u&\dfrac{x-\sqrt{x^2-4}}{2}}.$$ 
We have $\sqrt{x^2-4} \in \F_p$ if and only if ${\rm Im}\rho^R\subset \SL_2(\F_{p})$. 

The third author's representation ${\bs \rho}^U$ over $\F_{p^2}$ is given by 
$$\rho^U(a)=\spmx{\dfrac{\,x\,}{2}& 1\\ \dfrac{x^2-4}{4}&\dfrac{\,x\,}{2}}, \ \ 
\rho^U(b)=\spmx{\dfrac{\,x\,}{2}& \dfrac{-(1-v)^2u}{x^2-4}\\ \dfrac{-(1+v)^2u}{x^2-4}&\dfrac{\,x\,}{2}}.$$ 
Suppose in addition $x=y$. Then we have $\sqrt{x-2} \in \F_p$ if and only if ${\rm Im}\rho^U\subset \SL_2(\F_p)$. 

The True/False of $\sqrt{x^2-4} \in \F_p$ and $\sqrt{x-2} \in \F_p$ coincide if and only if $\sqrt{x+2}\in \F_p$. 
\end{eg}

\begin{eg} \label{eg.L=x-2}
When $n=-1$, we have $f_{-1}(x,y)=y^2-(x^2+1)y+2x^2-1$, $\dfrac{\der f_{-1}}{{\der y}}(x,y)=2y-(x^2+1)$, 
and $\dfrac{\der f_{-1}}{{\der x}}(x,y)=-2xy+4x=2x(2-y)$. 
Put $\beta=\dfrac{5+\sqrt{14}}{2}$, so that $\beta$ is a root of $f_{-1}(2,y)=y^2-5y+7$. 
Then $(2,\beta)$ corresponds to a lift $\rho_{\rm hol}^+$ of holonomy representation of the figure-eight knot $J(2,-2)$. 
Let $\ol{\rho}$ denote the residual representation of $\rho_{\rm hol}^+$ over $F$ with ${\rm char}=p>2$. 
Note that we have $\dfrac{\der f_{-1}}{\der y}(2,\beta) =\dfrac{\sqrt{-14}}{2}$ and $\dfrac{\der f_{-1}}{\der x}(2,\beta)=-2-\sqrt{-14}=\dfrac{18}{2-\sqrt{-14}}$. 
If $p\neq 7$, then we have $L_{\bs \rho}$ $\dot{=}$ $\Delta_{{\bs \rho},0}(1)$ $\dot{=}$ $x-2$ in $O[[x-2]]$. 
If $p\neq 3$, then we instead have 
$L_{\bs \rho}$ $\dot{=}$ $x_f(y)-2=\sqrt{\dfrac{y^2-y-1}{y-2}}-2$ $\dot{=}$ $y-\beta$ in $O[[y-\beta]]$.
\end{eg}

\begin{eg} 
Here we list for $n=-3,-2,-1,2,3$ all roots of $k_n$ in $\F_p$ for every $p\leq 41$ admitting a root with $|x|\leq 5$. 
(See also the table in Subsection \ref{subsec.tables}.)
\begin{multicols}{4}
\noindent 
\begin{tabular}{|l||l|l|} \hline 
$n$&$p$&$x$\\ \hline \hline
$-3$&5&$-1,-2$\\
&11&$-2,-3,4,5$\\
&19&$-3,4,5,6$\\ 
&29&$-4,-5,6,7$\\ 
&41&$-5,-6,7,8$\\ \hline 
\end{tabular}

\begin{tabular}{|l||l|l|} \hline 
$-2$&7&$-1$\\ 
&13&4\\
&29&$-2,-6,12$ \\
&41&$5,28,29$\\ \hline 
$-1$&$p$&1\\ \hline 
\end{tabular} 

\ \ \ \ \begin{tabular}{|l||l|l|} \hline
2&5&$-1$\\
&11&$-2,5$\\
&19&$-3,6$\\
&29&$-4,7$\\
&41&$-5,8$\\ \hline
\end{tabular}

\ \hspace{-7mm}
\begin{tabular}{|l||l|l|} \hline
3&3&1\\
&5&1\\
&7&$1,-2,4$\\
&17&$1,-5,7$\\
&23&$-1,4,6$\\ \hline 
\end{tabular}
\end{multicols}  

Note that $x=-1$ corresponds to an abelian residual representation. 
When $n=-3$ and $p=5$, then we have $k_{-3}=(x+1)^2(x-3)^2$ in $\F_5[x]$. 
Other cases in above are all single roots, hence we have $L_{\bs \rho}$ $\dot{=}$ $(x-\alpha)^2$ in $\Zp[[x-\alpha]]$ or $L_{\bs \rho}$ $\dot{=}$ $(y-\alpha)^2$ in $\Zp[[y-\alpha]]$. 
The list includes \cite[Example 4.5 (3)]{KMTT2018} (i) $n=2$, $p=11$, $x=5$ and (ii) $n=2$, $p=19$, $x=6$. 
\end{eg}

\begin{eg}
(i) 
When $n=-1$, then we have $k_{-1}=x-1$ for any $p$. 
Since $\dfrac{\der f_{-1}}{\der x}(1,1) =2 \neq 0$ and $\dfrac{\der f_{-1}}{\der y}(1,1)=0$ in $\F_p$, 
the implicit function $x_f(y) \in \Zp[[y-1]]$ is defined while $y_f(x) \in \Zp[[x-1]]$ does not, 
yielding a universal deformation $\bs{\rho}:\pi_{-1}\to \SL_2(\Zp[[y-1]])$. 
We have $L_{\bs \rho}$ $\dot{=}$ $(y-1)^2$ in $\Zp[[y-1]]$.  

(ii) When $n=-3,$ $p=11$, then 
we have $k_{-3}=-(x^2-3x+1)(x^2-x-1)=-(x+2)(x-5)(x+3)(x-4)$ in $\F_{11}[x]$. 
If $\ol{\alpha}=-2$. then 
$\dfrac{\der f_n}{\der x}(\ol{\alpha},\ol{\alpha})=2^5 7^2 11=0$,  
$\dfrac{\der f_n}{\der y}(\ol{\alpha},\ol{\alpha})=-23^1 389^1\neq 0$. 
In this case only $x_f(x)$ is defined, and we have $L_{\bs \rho}$ $\dot{=}$ $(x+2)^2$ in $\Z_{11}[[x+2]]$. 
\end{eg} 

\begin{eg} 
If $p\mid 3n-1$ and $n$ is odd, then we have $k_n(3)=\pm \dfrac{3n-1}{2}=0$ in $F$. 

(i) When $n=-3$ and $p=5$, we have $k_{-3}=-(x^2-3x+1)(x^2-x-1)=-(x+1)^2(x-3)^2$ in $\F_5[x]$. 
Put $O=\Z_5[\sqrt{5}]$. Then we have $x^2-x-1=(x+\dfrac{1-\sqrt{5}}{2})(x+\dfrac{1+\sqrt{5}}{2})$ in $O[x]$. 
If we take $\alpha=\dfrac{1-\sqrt{5}}{2}$, then we have $\dfrac{1+\sqrt{5}}{2}-\alpha=\sqrt{5} \in \mf{m}_O=(\sqrt{5})$. 
Since $f_y(3,3)\neq 0$, we have $y_f(x)$ and ${\bs \rho}:\pi_n\to \SL_2(O[[x-\alpha]])$ 
with $L_{\bs \rho}$ $\dot{=}$ $ (x^2-x-1)^2=(x-\alpha)^2(x-\alpha+\sqrt{5})^2=
(x-\alpha)^4+2\sqrt{5}(x-\alpha)^3+5(x-\alpha)^2$. 
We may also verify the equality $\dfrac{d^2L_{\bs \rho}}{dx^2}|_{x=\alpha}=5$ by 
substituting the value into the formula $\dfrac{d^2 \tau_n(x,y_f(x))}{dx^2}|_{x=\alpha}=\dfrac{2(3n-1)^2}{(\alpha-2)(n\alpha^2-2n\alpha-1)^2}$ 
in the proof of Lemma \ref{lem.order2}.

(ii) Similarly, when $n=5$ and $p=7$, we have $p\mid 3n-1$ and 
$k_5=(x^3 - 4 x^2 + 3 x + 1) (x^3 - 2 x^2 - x + 1)=(x+1)^3(x-3)^3 \in \F_7[x]$. 
By putting $O=\Z_7[\sqrt[3]{7}]$ and letting $\alpha \in O$ be a root of $k_5$ with the residue $\ol{\alpha}=3$, 
we have $L_{\bs \rho}$ $\dot{=}$ $(x^3-2x^2-x+1)^2$ in $O[[x-\alpha]]$. 
\end{eg} 

\section*{Acknowledgments} 
We would like to express our gratitude to L\'eo B\'enard, Shinya Harada, Teruhisa Kadokami, Tomoki Mihara, Masakazu Teragaito, Yoshikazu Yamaguchi, and anonymous referees for useful comments. 
The second author has been partially supported by a grant from the Simons Foundation (\#354595). 
The third author has been partially supported by JSPS KAKENHI Grant Number JP19K14538.

\bibliographystyle{amsalpha}
\bibliography{TangeTranUeki.arXiv.bbl} 



\end{document}